\documentclass[12pt]{amsart}
\usepackage[left=1in, right=1in, top=1in]{geometry}

\usepackage{etoolbox}

\makeatletter
\let\old@tocline\@tocline
\let\section@tocline\@tocline
\newcommand{\section@dotsep}{4.5}
\newcommand{\subsection@dotsep}{4.5}
\patchcmd{\@tocline}
  {\hfil}
  {\nobreak
     \leaders\hbox{$\m@th
        \mkern \section@dotsep mu\hbox{.}\mkern \section@dotsep mu$}\hfill
     \nobreak}{}{}
\let\section@tocline\@tocline
\let\@tocline\old@tocline

\patchcmd{\@tocline}
  {\hfil}
  {\nobreak
     \leaders\hbox{$\m@th
        \mkern \subsection@dotsep mu\hbox{.}\mkern \subsection@dotsep mu$}\hfill
     \nobreak}{}{}
\let\subsection@tocline\@tocline
\let\@tocline\old@tocline

\let\old@l@section\l@section
\let\old@l@subsection\l@subsection

\def\@tocwriteb#1#2#3{%
  \begingroup
    \@xp\def\csname #2@tocline\endcsname##1##2##3##4##5##6{%
      \ifnum##1>\c@tocdepth
      \else \sbox\z@{##5\let\indentlabel\@tochangmeasure##6}\fi}%
    \csname l@#2\endcsname{#1{\csname#2name\endcsname}{\@secnumber}{}}%
  \endgroup
  \addcontentsline{toc}{#2}%
    {\protect#1{\csname#2name\endcsname}{\@secnumber}{#3}}}%

\newlength{\@tocsectionindent}
\newlength{\@tocsubsectionindent}
\newlength{\@tocsubsubsectionindent}
\newlength{\@tocsectionnumwidth}
\newlength{\@tocsubsectionnumwidth}
\newlength{\@tocsubsubsectionnumwidth}
\newcommand{\settocsectionnumwidth}[1]{\setlength{\@tocsectionnumwidth}{#1}}
\newcommand{\settocsubsectionnumwidth}[1]{\setlength{\@tocsubsectionnumwidth}{#1}}
\newcommand{\settocsubsubsectionnumwidth}[1]{\setlength{\@tocsubsubsectionnumwidth}{#1}}
\newcommand{\settocsectionindent}[1]{\setlength{\@tocsectionindent}{#1}}
\newcommand{\settocsubsectionindent}[1]{\setlength{\@tocsubsectionindent}{#1}}
\newcommand{\settocsubsubsectionindent}[1]{\setlength{\@tocsubsubsectionindent}{#1}}

\renewcommand{\l@section}{\section@tocline{1}{\@tocsectionvskip}{\@tocsectionindent}{}{\@tocsectionformat}}%
\renewcommand{\l@subsection}{\subsection@tocline{1}{\@tocsubsectionvskip}{\@tocsubsectionindent}{}{\@tocsubsectionformat}}%
\renewcommand{\l@subsubsection}{\subsubsection@tocline{1}{\@tocsubsubsectionvskip}{\@tocsubsubsectionindent}{}{\@tocsubsubsectionformat}}%
\newcommand{\@tocsectionformat}{}
\newcommand{\@tocsubsectionformat}{}
\newcommand{\@tocsubsubsectionformat}{}
\expandafter\def\csname toc@1format\endcsname{\@tocsectionformat}
\expandafter\def\csname toc@2format\endcsname{\@tocsubsectionformat}
\expandafter\def\csname toc@3format\endcsname{\@tocsubsubsectionformat}
\newcommand{\settocsectionformat}[1]{\renewcommand{\@tocsectionformat}{#1}}
\newcommand{\settocsubsectionformat}[1]{\renewcommand{\@tocsubsectionformat}{#1}}
\newcommand{\settocsubsubsectionformat}[1]{\renewcommand{\@tocsubsubsectionformat}{#1}}
\newlength{\@tocsectionvskip}
\newcommand{\settocsectionvskip}[1]{\setlength{\@tocsectionvskip}{#1}}
\newlength{\@tocsubsectionvskip}
\newcommand{\settocsubsectionvskip}[1]{\setlength{\@tocsubsectionvskip}{#1}}
\newlength{\@tocsubsubsectionvskip}
\newcommand{\settocsubsubsectionvskip}[1]{\setlength{\@tocsubsubsectionvskip}{#1}}

\patchcmd{\tocsection}{\indentlabel}{\makebox[\@tocsectionnumwidth][l]}{}{}
\patchcmd{\tocsubsection}{\indentlabel}{\makebox[\@tocsubsectionnumwidth][l]}{}{}
\patchcmd{\tocsubsubsection}{\indentlabel}{\makebox[\@tocsubsubsectionnumwidth][l]}{}{}

\newcommand{\@sectypepnumformat}{}
\renewcommand{\contentsline}[1]{%
  \expandafter\let\expandafter\@sectypepnumformat\csname @toc#1pnumformat\endcsname%
  \csname l@#1\endcsname}
\newcommand{\@tocsectionpnumformat}{}
\newcommand{\@tocsubsectionpnumformat}{}
\newcommand{\@tocsubsubsectionpnumformat}{}
\newcommand{\setsectionpnumformat}[1]{\renewcommand{\@tocsectionpnumformat}{#1}}
\newcommand{\setsubsectionpnumformat}[1]{\renewcommand{\@tocsubsectionpnumformat}{#1}}
\newcommand{\setsubsubsectionpnumformat}[1]{\renewcommand{\@tocsubsubsectionpnumformat}{#1}}
\renewcommand{\@tocpagenum}[1]{%
  \hfill {\mdseries\@sectypepnumformat #1}}

\let\oldappendix\appendix
\renewcommand{\appendix}{%
  \leavevmode\oldappendix%
  \addtocontents{toc}{%
    \protect\settowidth{\protect\@tocsectionnumwidth}{\protect\@tocsectionformat\sectionname\space}%
    \protect\addtolength{\protect\@tocsectionnumwidth}{2em}}%
}
\makeatother

\makeatletter
\settocsectionnumwidth{2em}
\settocsubsectionnumwidth{2.5em}
\settocsubsubsectionnumwidth{3em}
\settocsectionindent{1pc}%
\settocsubsectionindent{\dimexpr\@tocsectionindent+\@tocsectionnumwidth}%
\settocsubsubsectionindent{\dimexpr\@tocsubsectionindent+\@tocsubsectionnumwidth}%
\makeatother

\settocsectionvskip{0pt}
\settocsubsectionvskip{0pt}
\settocsubsubsectionvskip{0pt}

\usepackage{bbding}
\usepackage{mathrsfs}
\usepackage{mathrsfs}
\usepackage{amsfonts}
\usepackage{cases}
\usepackage{latexsym}
\usepackage{amsmath}
\usepackage[all]{xy}
\usepackage{stmaryrd}
\usepackage{amsmath,amssymb,amscd,bbm,amsthm,mathrsfs,dsfont}
\usepackage{color}
\usepackage{fancyhdr}
\usepackage{amsxtra,ifthen}
\usepackage{verbatim}

\usepackage{fancyhdr}
\usepackage{amsxtra,ifthen}
\usepackage{verbatim}

\usepackage[numbers,sort&compress]{natbib}
\usepackage[pagebackref]{hyperref}
\usepackage{hypernat}

\usepackage{mathrsfs}
\usepackage{mathrsfs}
\usepackage{amsfonts}
\usepackage{comment}
\usepackage{cases}
\usepackage{latexsym}
\usepackage{amsmath}
\usepackage[all]{xy}
\usepackage{stmaryrd}
\usepackage{amsfonts}
\usepackage{amsmath,amssymb,amscd,bbm,amsthm,mathrsfs,dsfont}

\DeclareRobustCommand\longtwoheadrightarrow
     {\relbar\joinrel\twoheadrightarrow}

\usepackage{tikz}

\usepackage{pgflibraryarrows}
\usepackage{pgflibrarysnakes}

\usepackage[numbers,sort&compress]{natbib}
\usepackage[pagebackref]{hyperref}
\usepackage{hypernat}

\usepackage{color, hyperref}
\definecolor{blue}{rgb}{0.9,0.0,0.9}
\hypersetup{colorlinks, breaklinks,
            linkcolor=red,urlcolor=blue,
            anchorcolor=blue,citecolor=blue}

\usepackage{fancyhdr}
\usepackage{amsxtra,ifthen}
\usepackage{verbatim}

\usepackage{tikz,xcolor,hyperref}
\definecolor{lime}{HTML}{A6CE39}
\DeclareRobustCommand{\orcidicon}{
\begin{tikzpicture}
\draw[lime, fill=lime] (0,0)
circle[radius=0.16]
node[white]{{\fontfamily{qag}\selectfont \tiny \.{I}D}}; 
\end{tikzpicture}
\hspace{-2mm}
}
\foreach \x in {A, ..., Z}{%
\expandafter\xdef\csname orcid\x\endcsname{\noexpand\href{https://orcid.org/\csname orcidauthor\x\endcsname}{\noexpand\orcidicon}}
}

\usepackage{enumitem}

\theoremstyle{plain}
\newtheorem{theorem}{Theorem}[section]
\newtheorem{lemma}[theorem]{Lemma}
\newtheorem{proposition}[theorem]{Proposition}
\newtheorem{hypothesis}[theorem]{HYP}
\newtheorem{corollary}[theorem]{Corollary}

\theoremstyle{definition}

\newtheorem{remark}[theorem]{Remark}

\newtheorem{question}[theorem]{Question}

\allowdisplaybreaks[4]

\begin{document}

\title[Multivariable $(\varphi ,\Gamma )$-modules and Representations of Products of Galois Groups]
{Multivariable $(\varphi ,\Gamma )$-modules and Representations of Products of Galois Groups:\\
 The Case of Imperfect Residue Field}

\author{Jishnu Ray \hspace{-1.5mm}\orcidA{}}
\address{Ray: School of Mathematics, Tata Institute of Fundamental Research, Homi Bhabha Road, Colaba, Mumbai 400005, India}
\email{jishnuray1992@gmail.com}\email{jishnu.ray@tifr.res.in}

\author{Feng Wei* \hspace{-1.5mm}\orcidB{}}\thanks{*Corresponding author}
\address{Wei (Corresponding Author): School of Mathematics and Statistics, Beijing
Institute of Technology, Beijing, 100081, China}
\email{daoshuo@hotmail.com} \email{daoshuowei@gmail.com}

\author{Gergely Z\'abr\'adi \hspace{-1.5mm}\orcidC{}}
\address{Z\'abr\'adi: Institute of Mathematics, E\"otv\"os Lor\'and University, P\'azm\'any P\'eter s\'et\'any 1/C, H-1117 Budapest, 
Hungary and MTA Lend\"ulet Automorphic Research Group}

\email{gergely.zabradi@ttk.elte.hu}

\begin{abstract}
Let $K$ be a complete discretely valued field with mixed characteristic $(0, p)$ and imperfect 
residue field $k_\alpha$. Let $\Delta$ be a finite set. We construct an equivalence of categories between finite dimensional 
$\Bbb{F}_p$-representations of the product of $\Delta$ copies of the absolute Galois group of $K$ and multivariable \' etale 
$(\varphi, \Gamma)$-modules over a multivariable Laurent series ring over $k_\alpha$. \\

\vspace{0.1mm}

\noindent{\textit {R\'esum\'e} ($(\varphi,\Gamma)$-modules multivariables et repr\'esentations du produit du groupe de Galois: le cas des corps r\'esiduels imparfaits)}\\
   Soit $K$ un corps discr\`etement valu\'e \`a charact\'eristique mixte $(0, p)$ et un corps 
r\'esiduel imparfait $k_\alpha$. Soit $\Delta$ un ensemble fini. Nous \'etablissons une \'equivalence de cat\'egories 
entre des repr\'esentations de dimensions finies sur $\mathbb{F}_p$ du produit de $\Delta$ copies du groupe 
absolu de Galois de $K$ et des $(\varphi, \Gamma)$-modules \'etales multivariables sur un anneau multivariable 
des s\'eries Laurent sur $k_\alpha$.
\end{abstract}



\thanks{}

\date{\today}

\maketitle


\section{Introduction}\label{xxsec1}

\subsection{Motivation of this work}\label{xxsec1.1}

Fontaine's theory of $(\varphi, \Gamma)$-modules is a fundamental tool to describe 
and classify continuous representations of the Galois 
group of a finite extension of $\Bbb{Q}_p$ on a finite-dimensional $\Bbb{Q}_p$-vector space. 
With the help of Fontaine's theory of $(\varphi, \Gamma)$-modules, one can understand the 
$p$-adic and mod-$p$ Langlands correspondence in the 
case of the general linear group ${\rm GL}_2$ over the field $\Bbb{Q}_p$ of $p$-adic numbers, see 
\cite{Breuil1, Breuil2, Breuil3, BreuilPaskunas, Colmez1, Colmez2, Colmez3, Paskunas}. 
By invoking the theory of $(\varphi, \Gamma)$-modules, the $p$-adic and mod-$p$ 
representations of ${\rm GL}_2(\Bbb{Q}_p)$ can be connected with $p$-adic and mod-$p$ Galois representations of $\Bbb{Q}_p$. 
To extend the correspondence to other $p$-adic reductive groups beyond ${\rm GL}_2(\Bbb Q_p)$, 
one naturally wants to generalize Fontaine's theory of $(\varphi,\Gamma)$-modules. 
There have been conjectural progress in attempts to generalize $p$-adic Langlands 
beyond ${\rm GL}_2(\Bbb{Q}_p)$ along these lines; two kinds of multivariable version of 
$(\varphi, \Gamma)$-modules can be found in the literature. Berger's multivariable $(\varphi, \Gamma)$-modules 
is an attempt to generalize $p$-adic Langlands for ${\rm GL}_2(F)$, where $F$ is a finite extension of 
$\Bbb{Q}_p$ \cite{Berger3, Berger4}. The third author of this current work also defines multivariable 
$(\varphi, \Gamma)$-module over a $m$-variable Laurent series ring in an attempt to 
generalize $p$-adic Langlands for ${\rm GL}_m(\Bbb{Q}_p)$ \cite{SchneiderVignerasZabradi,Zabradi1,Zabradi2}. One might also try to look at 
Z\' abr\' adi's multivariable $(\varphi, \Gamma)$-modules over Lubin-Tate extension to conjecturally understand $p$-adic Langlands for ${\rm GL}_m(F)$ \cite{Grosse-Klonne3}.
It has become clear that essentially all of $p$-adic Hodge theory can be formulated in terms of 
$(\varphi, \Gamma)$-modules; moreover, this formulation has driven much recent progress in the 
subject and powered some notable applications in arithmetic geometry \cite{BrinonConrad}. See \cite{Kedlaya} for a quick 
introduction to this circle of ideas or \cite{Schneider} for a more in-depth treatment. Multivariable $(\varphi,\Gamma)$-modules are also related \cite{Kedlaya2, CarterKedlayaZabradi} to Scholze's theory of perfectoid spaces.

This paper can be considered as a complement to the third author's independent work \cite{Zabradi2}
in which he shows that the category of continuous representations of the $m^{\rm th}$ direct product 
of the absolute Galois group of $\Bbb{Q}_p$ on finite dimensional $\Bbb{F}_p$-vector spaces (resp.\ $\Bbb{Z}_p$-modules, resp.\ $\Bbb{Q}_p$-vector spaces)
is equivalent to the category of \' etale multivariable $(\varphi, \Gamma)$-modules over a certain $m$-variable 
Laurent series ring over $\Bbb{F}_p$ (resp.\ over $\Bbb{Z}_p$, resp.\ over $\Bbb{Q}_p$). In the current paper, we are going to extend this equivalence 
of categories for continuous $\Bbb{F}_p$-representations of the $m^{\rm th}$ direct product of the absolute 
Galois group of a complete discretely valued field $K$ with mixed characteristic $(0, p)$ whose 
residue field $k_\alpha$ is imperfect and has a finite $p$-basis, i.e. $[k_\alpha: k_\alpha^p]=p^d$ (for some $d\geq 1$). We plan to come back to the question of $p$-adic representations in the future. We expect applications of our results to $p$-adic Hodge theory of products of varieties over $p$-adic fields. To 
state our main Theorem (Theorem \ref{Them5.14}) precisely, we need to review the third author's work on multivariable $(\varphi, \Gamma)$-modules 
\cite{Zabradi2} and his main theorem. 

\subsection{Z\' abr\' adi's work \cite{Zabradi2}}\label{xxsec1.2}
Let $F$ be a finite extension of $\Bbb{Q}_p$ with residue field $k_F$ (which is perfect). 
For a finite set $\Delta$,  let $\mathcal{G}_{{\Bbb{Q}_p}, \Delta}: = \prod_{\alpha  \in \Delta } {{\rm{Gal}}(\overline{{\Bbb{Q}_p}}/{\Bbb{Q}_p})} $ 
denote the direct power of the absolute Galois group of $\Bbb{Q}_p$ indexed by $\Delta$. We 
denote by ${\rm Rep}_{k_F}(\mathcal{G}_{\Bbb{Q}_p, \Delta})$ the category of continuous representations 
of the profinite group $\mathcal{G}_{\Bbb{Q}_p,\Delta}$ on finite dimensional $k_F$-vector spaces. 
For independent commuting variables ${X_\alpha }\, (\alpha  \in \Delta )$,
we write
$$
{E_{\Delta, k_F }}  : =  k_F \left[\kern-0.15em\left[ {{X_\alpha }|\alpha  \in \Delta }
 \right]\kern-0.15em\right][X_\Delta^{-1}],
$$
where $X_\Delta=\prod_{\alpha\in \Delta} X_\alpha$. For each element $\alpha  \in \Delta$,  
we have the partial Frobenius ${\varphi _\alpha }$, and group ${G_{K_\alpha} } \cong {\rm{Gal}}({\Bbb{Q}_p}({\mu _{{p^\infty }}})/{\Bbb{Q}_p})$ 
acting on the variable ${X_\alpha }$ in the usual way and commuting with the other 
variables ${X_\beta }\, (\beta  \in \Delta \backslash \{ \alpha \} )$ in the ring $E_{\Delta, k_F}$ 
(some authors also write $G_{K_\alpha}$ as $\Gamma_\alpha$).  
A $({\varphi_\Delta },  \Gamma_\Delta)$-\textit{module} (or a $(\varphi_\Delta, G_\Delta)$-module)  over ${E_{\Delta, k_F}}$ is a finitely generated 
${E_{\Delta, k_F}}$-module
$D$ together with commuting semilinear actions of the operators ${\varphi _\alpha }$ and groups ${G_{K_\alpha}}\, (\alpha  \in \Delta )$. 
We say that $D$ is \textit{\'{e}tale} if the map id $ \otimes {\varphi _\alpha }:\varphi _\alpha ^*D \longrightarrow  D$
is an isomorphism for all $\alpha  \in \Delta $.  Then the third author independently shows that  
${\rm Rep}_{k_F}(\mathcal{G}_{\Bbb{Q}_p, \Delta})$  is equivalent to the category of \' etale $(\varphi_\Delta, G_\Delta)$-modules over $E_{\Delta, k_F}$.

\subsection{Andreatta's work \cite{Andreatta} and Scholl's work \cite{Scholl}} Let 
us review Scholl's work \cite{Scholl} and parts of Andreatta's work \cite{Andreatta} where 
they work with single variable classical $(\varphi, \Gamma)$-module but over an imperfect 
residue field. Let $K$ be a complete discretely valued field (with uniformizer $p$) 
of mixed characteristic $(0, p)$ with imperfect residue field $k_K$ having a $p$-basis, 
i.e. $[k_K: k_K^p]=p^d$. Let $t_1, t_2, \cdots, t_d\in K$ be a lift of a $p$-basis $\overline{t_1}, \overline{t_2}, \cdots, \overline{t_d}$ 
of $k_K$. Define $
K_\infty=\bigcup_n\, K(\mu_{p^n}, t_1^{1/p^n}, \cdots, t_d^{1/p^n}) 
$, $G_K={\rm Gal}(K_\infty/K)$ and $\mathcal{G}_K={\rm Gal}(\overline{K}/K)$. Note that, in contrast with the perfect residue field case, $G_K$ is not abelian. Scholl \cite{Scholl} and Andreatta \cite{Andreatta} defined a field of norms $E_K$ for $K$, 
and have shown that $E_K\cong k_K((\overline{\pi}))$, where $\varepsilon=\overline{\pi}+1\in E_K$ is compatible system of $p$-power roots of unity in $K_\infty$ (cf. \cite[Section 2.3]{Scholl}). 
Finally, Andreatta \cite[Theorem 7.11]{Andreatta} showed that ${\rm Rep}_{{\Bbb F}_p}(\mathcal{G}_K)$ 
is equivalent to the category of (single variable, i.e. classical) \' etale $(\varphi, G_K)$-mdule over $E_K$.

\subsection{Our work in this paper}\label{xxsec1.4} In this paper, we will extend 
Scholl's and Andreatta's result to the case of multivariable 
$(\varphi_\Delta, G_\Delta)$-modules over an imperfect residue field. 
Precisely speaking, for a finite set $\Delta$ and a collection of possibly distinct fields $K_\alpha$ as above, we define 
$$
G_\Delta=\prod_{\alpha\in \Delta} G_{K_\alpha},
$$
$$
\mathcal{G}_\Delta=\prod_{\alpha\in \Delta} \mathcal{G}_{K_\alpha}, 
$$
and the Laurent series ring
$$
E_\Delta: =  (\bigotimes_{\alpha\in\Delta}k_{K_\alpha}) \left[\kern-0.15em\left[ {{X_\alpha }|\alpha  \in \Delta }
 \right]\kern-0.15em\right][X_\Delta^{-1}].
$$
It should be remarked that for each $\alpha$, $G_{K_\alpha}\cong \Gamma_\alpha\ltimes H_\alpha$, 
where $\Gamma_\alpha\cong {\rm Gal}(K(\mu_{p^\infty})/K)$ and $H_\alpha\cong {\rm Gal}(K_\infty/K(\mu_{p^\infty}))$ 
and so $G_{K_\alpha}$ is a noncommutative $p$-adic Lie group. Extending actions of \cite{Scholl}, we 
provide the ring $E_\Delta$ with the natural actions of partial Frobenius $\varphi_\alpha\, (\alpha\in \Delta)$, 
absolute Frobenius $\varphi_s=\prod \limits_{\alpha\in \Delta}$ and the Galois group $G_\Delta$. 
We define the category $\mathcal{D}^{\rm et}(\varphi_\Delta, G_\Delta, E_\Delta)$ 
of multivariable \' etale $(\varphi_\Delta, G_\Delta)$-modules over $E_\Delta$ in Section \ref{xxsec3.3}. Our main Theorem 
(see Theorem \ref{Them5.14}) is that there is an equivalence of categories between 
${\rm Rep}_{\Bbb{F}_p}(\mathcal{G}_\Delta)$ and  $\mathcal{D}^{\rm et}(\varphi_\Delta, G_\Delta, E_\Delta)$. 
Fortunately, many arguments in the proofs given by the third author in \cite{Zabradi2} (in the perfect residue field case) 
can be generalized and adapted to the case when the residue field is imperfect. Therefore, 
our proofs will mostly follow the line of arguments given in \cite{Zabradi2} with modifications 
when necessary, invoking the results of Andreatta and Scholl (for the single variable, imperfect 
residue field case) and using induction. 


\section{Kummer Towers}\label{xxsec2}
In this Section, we will introduce the Iwasawa theoretic tower that we 
are going to work with. Let $L$ be a complete discretely valued field of mixed characteristic 
$(0, p)$. Suppose that $[k_L: k_L^p]=p^d$, where $k_L$ is the residue field of $L$. 
Let us choose a complete subfield $K$ of $L$ with the same residue field $k_L$ in 
which $p$ is an uniformizer (the existence of such a subfield is proved in \cite[Page 211-212]{Matsumura}). 
Let  $t_1, t_2, \cdots, t_d\in L$ be a lift of a $p$-basis $\overline{t_1}, \overline{t_2}, \cdots, \overline{t_d}$ 
of $k_L$. For $n\geq 1$, define $
K_n= K(\mu_{p^n}, t_1^{1/p^n}, \cdots, t_d^{1/p^n}) 
$, $K_\infty=\bigcup_n K_n$, $L_n=LK_n$ and $L_\infty=LK_\infty$. 
Define the Galois groups $\Gamma_L={\rm Gal}(L(\mu_{p^\infty})/L)$, 
 $G_L={\rm Gal}(L_\infty/L)$, $H_L={\rm Gal}(L_\infty/L(\mu_{p^\infty}))$, 
 $\mathcal{H}_L={\rm Gal}(\overline{K}/L_\infty)$, $\mathcal{G}_L={\rm Gal}(\overline{K}/L)$. We identify $\Gamma_L$ 
 via the quotient map with the subgroup ${\rm Gal}(L_\infty/L_\infty^\prime)$ of $G_L$, where 
$L_\infty=\varinjlim\limits_{n} L(t_1^{1/p^n}, \cdots, t_d^{1/p^n})$.

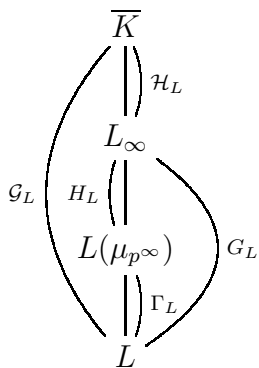
$$
\xymatrix{
\overline{K}\ar@{-}[d]\ar@{-}@/^/[d]^{\mathcal{H}_L}\ar@/_30pt/@{-}[ddd]_{\mathcal{G}_L}\\
L_\infty\ar@{-}[d]\ar@{-} @/_/[d]_{H_L} \ar@/^35pt/@{-}[dd]^{G_L} \\
L(\mu_{p^\infty})\ar@{-}[d]\ar@{-}@/^/[d]^{\Gamma_L}\\
L}
$$

Note that the cyclotomic character $\chi$ identifies $\Gamma_L$ with an open subgroup of $\Bbb{Z}_p^{\times}$. 
We also have that $G_L\cong \Gamma_L\ltimes H_L$, where $H_L\cong \Bbb{Z}_p^d$ and $G_L$ 
is a non-commutative $p$-adic Lie group of dimension $d+1$.The tower $(K_n)_{n\geq 1}$ is strictly 
deeply ramified in the sense of \cite{Scholl}. By \cite[Section 1.3]{Scholl}, we can say that 
there exists $n_0\in \Bbb{N}$ and $\xi\in \mathcal{O}_{K_{n_0}}$ satisfying $0<|\xi|_K< 1$ such that 
for all $n\geq n_0$, the $p$-power map $\mathcal{O}_{K_{n+1}}/(\xi)\longrightarrow \mathcal{O}_{K_n}/(\xi)$ 
is a surjection. We denote $E^+_K=\varprojlim\limits_{n\geq n_0} \mathcal{O}_{K_n}/(\xi)$, where 
the inverse limit is taken with respect to the $p$-power maps. Then $E^+_K$ is a complete,  discretely  valued
ring of characteristic $p$, independent of $n_0$ and $\xi$ (cf. \cite[Section 2.1]{Zerbes}). Let $E_K$ 
be the fraction field of $E^+_K$. We call $E_K$ the \textit{field of norms} of the tower $K_n$. 
Note that $E_K$ has a natural action of $G_K$ that commutes with the Frobenius operator $\varphi$.

For every finite extension $K^\prime$ of $K$, $E_{K^\prime}$ is a finite separable extension of $E_K$. 
Let $E_K^{\rm sep}=\bigcup_{K^\prime} E_{K^\prime}$. It follows from \cite[Corollary 6.4]{Andreatta} that 
there is an isomorphism of topological groups 
$$
{\rm Gal}(E_K^{\rm sep}/E_K)\cong {\rm Gal}(\overline{K}/K_\infty)=\mathcal{H}_K.
$$
We therefore conclude that $(E_K^{\rm sep})^{\mathcal{H}_K}\cong E_K$. 

\begin{theorem}{\rm \cite[Section 2.3]{Scholl}}\label{Them0.1}
The field of norms $E_K\cong k_K (( \overline{\pi} ))$, where $\varepsilon=1+\overline{\pi}$ is a compatible system of $p$-power roots of unity.
\end{theorem}

The field of norms $E_L=(E_K^{sep})^{\mathcal{H}_L}$ for the tower $L_\infty$ is a finite separable extension of $E_K$ of the form $E_L\cong k_K((X))$ for some noncanonical choice of uniformizer $X$. Therefore the action of $G_{L}$ on $E_L$ does not have an intrinsic description in general. We shall further make the following hypothesis on $K$ (hence on $L$, as we have $k_K=k_L$).
\begin{hypothesis}\label{hyp2.2}
The residue field $k_K$ is a finitely generated field extension of $\Bbb{F}_p$.
\end{hypothesis}

Let $\Delta$ be a finite set and we pick a complete discretely valued field $L=L_\alpha$ as above for each $\alpha$. We also allow the cardinality $d_\alpha$ of a $p$-basis to vary for $\alpha\in\Delta$. We put $k_\alpha:=k_{K_\alpha}=k_{L_\alpha}$ and assume that it satisfies HYP \ref{hyp2.2} for all $\alpha\in\Delta$. Further, let $\Bbb{F}_\alpha$ denote the maximal algebraic extension of $\Bbb{F}_p$ contained in $k_\alpha$. Then $\Bbb{F}_\alpha=k_\alpha((X_\alpha))^{G_{L,\Delta}}$ is a finite field and $k_\alpha$ is a finite separable extension of the function field $k_{\alpha,0}:=\Bbb{F}_\alpha(\overline{t_{\alpha,1}},\dots,\overline{t_{\alpha,1}})$ where $\overline{t_{\alpha,1}},\dots,\overline{t_{\alpha,1}}\in k_\alpha$ is a finite $p$-basis. We put $\Bbb{F}_\Delta:=\bigotimes_{\Bbb{F}_p,\alpha\in\Delta}\Bbb{F}_\alpha$ and $k_{\Delta,0}:=\bigotimes_{\Bbb{F}_p,\alpha\in\Delta}k_{\alpha,0}$.

\begin{lemma}\label{lem2.3}
Assume HXP \ref{hyp2.2} for each residue field $k_\alpha$. Then the $|\Delta|$-fold tensor product $k_\Delta:=\bigotimes_{\Bbb{F}_p,\alpha\in\Delta}k_\alpha$ is noetherian and regular (and, in particular, reduced). Further, for each $\alpha\in\Delta$ the relative Frobenius $\varphi_\alpha$ is injective on $k_\Delta$.
\end{lemma}
\begin{proof}
The ring $k_\Delta$ is the localization of a finitely generated $\Bbb{F}_p$-algebra, therefore it is noetherian. Since $\Bbb{F}_p$ is a field, any $\Bbb{F}_p$-module is flat. Now $\varphi_\alpha\colon k_\alpha\to k_\alpha$ is injective and on $k_{\Delta\setminus\{\alpha\}}$ it is the identity, therefore $\varphi_\alpha$ is also injective on $k_\Delta=k_\alpha\otimes_{\Bbb{F}_p}k_{\Delta\setminus\{\alpha\}}$. In particular, the absolute Frobenius $\varphi_s=\prod_{\alpha\in\Delta}\varphi_\alpha$ is also injective on $k_\Delta$, ie.\ $k_\Delta$ is reduced. The statement on the regularity follows from \cite[Theorem 1.6(c),(e)]{TousiYassemi} since $\Bbb{F}_p$ is perfect and $k_\Delta$ is noetherian.
\end{proof}

Since $\Bbb{F}_\Delta$ is a tensor product of finite fields, it is finite \'etale algebra over $\Bbb{F}_p$. Moreover, it has primitive idempotents $b_1,\dots, b_\ell\in \Bbb{F}_\Delta$ with $1=b_1+\cdots+b_\ell$ and $b_j\Bbb{F}_\Delta\cong \Bbb{F}_{p^f}$ where $f=\gcd(|\Bbb{F}_\alpha:\Bbb{F}_p| \mid \alpha\in\Delta)$, $1\leq j\leq \ell$. Note that for each $\alpha\in\Delta$ and $1\leq j\leq \ell$ the element $\varphi_\alpha(b_j)$ is also a primitive idempotent in $\Bbb{F}_\Delta$ and $\varphi_s(b_j)=b_j^p=b_j$. The quotient monoid $\Phi:=(\prod_{\alpha\in\Delta}\varphi_\alpha^{\Bbb{N}})/\varphi_s^{\Bbb{N}}$ is a group acting on the set of primite idempotents.

\begin{lemma}\label{Lem2.4}
The group $\Phi$ acts transitively on the set $\{b_1,\dots,b_\ell\}$ of primitive idempotents in $\Bbb{F}_\Delta$.
\end{lemma}
\begin{proof}
By induction on $|\Delta|$ we are reduced to the case $\Delta=\{\alpha,\beta\}$ has two elements. Put $p^n:=|\Bbb{F}_\alpha|$, $p^m:=|\Bbb{F}_\beta|$, and $f:=\gcd(n,m)$. Writing $\Bbb{F}_{p^m}=\Bbb{F}_p[X]/(g(X))$ for some irreducible monic polynomial $g(X)\in \Bbb{F}_p[X]$ we may write $g(X)=\prod_{i=1}^{f}g_i(X)$ over $\Bbb{F}_{p^n}$ so we have $\Bbb{F}_{p^n}\otimes_{\Bbb{F}_p}\Bbb{F}_{p^m}=\bigoplus_{i=1}^f\Bbb{F}_{p^n}[X]/(g_i(X))$ where $\varphi_\beta$ acts trivially on $\Bbb{F}_{p^n}$ and satisfies $\varphi_\beta(g_i)=g_{i+1}$ with the convention $g_{f+1}=g_1$.
\end{proof}

\section{Multivariable $(\varphi, \Gamma)$-modules}\label{xxsec3}

\subsection{} Let $\Delta$ be a finite set (which can be simple roots in the Lie algebra of a reductive group 
over $\Bbb{Z}_p$) and $(L_\alpha)_{\alpha\in\Delta}$ be a collection of complete discretely valued fields with residue fields $k_\alpha$ such that $k_\alpha$ is a finitely generated extension of $\Bbb{F}_p$ with finite $p$-basis $\overline{t_{\alpha,1}},\dots,\overline{t_{\alpha,1}}$. We further choose a complete subfield $K_\alpha\leq L_\alpha$ in which $p$ is a uniformizer and has the same residue field $k_\alpha$. Let us define 

\begin{enumerate}
\item[i)] $G_{L, \Delta}:=\prod \limits_{\alpha\in \Delta} G_{L_\alpha}$,
\item[ii)] $\mathcal{G}_{L, \Delta}:=\prod \limits_{\alpha\in \Delta} \mathcal{G}_{L_\alpha}$.
\end{enumerate}
We denote ${\rm Rep}_{\Bbb{F}_p}(\mathcal{G}_{L, \Delta})$ the category of continuous representations of the 
profinite group $\mathcal{G}_{L, \Delta}$ on finite dimensional $\Bbb{F}_p$-vector spaces. (In the future, 
$G_{K, \Delta}$ and $\mathcal{G}_{K, \Delta}$  will simply be denoted as $G_\Delta$ and $\mathcal{G}_\Delta$, dropping 
the subscript $K$.)

\subsection{Some Laurent series rings}\label{xxsec3.2} 
Consider the Laurent series $E_{L,\Delta}=E_{L,\Delta}^+[X_\Delta^{-1}]$ where $$E^+_{L,\Delta}=k_\Delta \llbracket X_\alpha|  \alpha\in \Delta \rrbracket=(\bigotimes_{\Bbb{F}_p,\alpha\in\Delta}k_\alpha)\llbracket X_\alpha|  \alpha\in \Delta \rrbracket$$
is the completed tensor product of $E_\alpha^+\, (\cong k_\alpha \llbracket X_\alpha \rrbracket)$ over $\Bbb{F}_p$ for all $\alpha\in \Delta$. 
Here $E^+_\alpha$ is the ring of integers of the field of norms $E_\alpha\, (\cong k_\alpha(( X_\alpha )))$ 
corresponding to $\alpha$. Here $X_\Delta:=\prod \limits_{\alpha\in \Delta} X_\alpha\in E^+_\Delta$. For each 
$\alpha$, we define the action of the partial Frobenius $\varphi_\alpha$ and the group $G_{L_\alpha}$ as follows (cf. \cite[Page 707, Section 2.3]{Scholl}). 
$$
{\varphi _\alpha }({X_\beta }): =
\begin{cases} 
{X_\beta } & {\rm{if}} \ \beta  \in \Delta \backslash \{ \alpha \},  \\{(1+{X_\alpha })^p}-1 = X_\alpha ^p& {\rm{if}} \ \beta =\alpha.
\end{cases} \eqno{(3.2.1)}
$$
$\varphi_\alpha$ acts on the coefficients in $k_\alpha$ as the $p$-th power map and as identity on $k_\beta$ for $\beta\in\Delta\setminus\{\alpha\}$. By the construction of the field of norms \cite{Scholl}, the group $G_{L_\alpha}$ acts on the ring $k_\alpha((X_\alpha))$ and we extend this action to $E_{L,\Delta}$ by acting trivially on $k_\beta((X_\beta))$ ($\beta\in\Delta\setminus\{\alpha\}$). We only describe the action of $G_{L_\alpha}$ on $E_{L,\Delta}$ intrinsically in case $L_\alpha=K_\alpha$. We put $\overline{\pi}_\alpha:=\varepsilon-1\in k_\alpha((X_\alpha))\subset \varprojlim\mathcal{O}_{L_n}/(\xi)$ where $\varepsilon=(\varepsilon_n)_{n\geq 0}$ is a compatible system of $p$-power roots of unity (with $\varepsilon_0=1\neq\varepsilon_1$). In case $L_\alpha=K_\alpha$ is unramified, we have $k_\alpha((X_\alpha))=k_\alpha((\overline{\pi}_\alpha))$, but in general $k_\alpha((X_\alpha))$ is a finite separable extension of $k_\alpha((\overline{\pi}_\alpha))$ of degree $e_L$ which is the absolute ramification index of $L$. Action of $G_{K_\alpha}=G_\alpha\cong \Gamma_\alpha\ltimes H_\alpha$; $H_\alpha\cong \Bbb{Z}_p^d$ on $E_\Delta:=E_{K,\Delta}$
can be described as follows. For any $\gamma_{\alpha}\in \Gamma_\alpha$ we have
$$
{\gamma _{\alpha} }({\overline{\pi}_\beta }): = \begin{cases} {\overline{\pi}_\beta } & {\rm{if}} \ \beta  \in \Delta \backslash \{ \alpha \},  \\ (1+{\overline{\pi}_\alpha })^{\chi(\gamma_\alpha)} - 1& {\rm{if}} \ \beta  = \alpha. \end{cases}\eqno{(3.2.2)}
$$
$\Gamma_\alpha$ acts as identity on $k_\alpha$. Let $\underline{b}_\alpha$ be the image of $\delta_{\alpha, \underline{b}}\in H_\alpha$ 
in $\Bbb{Z}_p^d$ and let $b_{\alpha, i}$ be the $i$-th component of $\underline{b}_\alpha$. Then, we have
$$
\delta_{\alpha, \underline{b}}(\overline{\pi}_\beta): =\overline{\pi}_\beta\, \,  \text{for all}\, \, \, \beta\, \, (\text{equal to}\, \,  \alpha \, \, \text{and}\, \, \text{also not equal to} \, \, \alpha), \eqno{(3.2.3)}
$$
$$
\delta_{\alpha, \underline{b}}(\overline{t_{\alpha,i}}): =(1+\overline{\pi}_\alpha)^{b_{\alpha, i}} \overline{t_{\alpha,i}}\, \,  \text{for all}\, \, 1\leq i \leq d_\alpha\ , \eqno{(3.2.4)}
$$
and $\delta_{\alpha, \underline{b}}$ acts as identity on $k_\beta$ for $\beta\neq\alpha$. Note that such an automorphism 
$\delta_{\alpha, \underline{b}}$ of $E_\Delta^+$ is unique which is easy to see from the case $|\Delta|=1$ in \cite[Page 707]{Scholl}.

Note that the absolute Frobenius $\varphi_s: E_{L,\Delta}^+\longrightarrow E^+_{L,\Delta}$  
equals the composite $\prod \limits_{\alpha\in \Delta}\varphi_\alpha$ of the partial Frobenii. Further, the actions of $\varphi_\alpha$ ($\alpha\in\Delta$) and $G_\beta$ ($\beta\in\Delta$) all commute with each other (even though the individual factors $G_\beta$ are non-abelian).



The ring $E_{L,\Delta}^+$ is noetherian and reduced by Lemma \ref{lem2.3}.

\subsection{Multivariable $(\varphi_\Delta, G_{L,\Delta})$-modules}\label{xxsec3.3}
By a $(\varphi_\Delta, G_{L,\Delta})$-module over $E_{L,\Delta}$, we mean a finitely 
generated module $D$ over $E_{L,\Delta}$ together with commuting semilinear actions of $\varphi_\alpha$
and the Galois groups $G_{L_\alpha}$ for all $\alpha\in \Delta$. By an \' etale $(\varphi_\Delta, G_{L,\Delta})$-module over $E_{L,\Delta}$, we mean a
$(\varphi_\Delta, G_{L,\Delta})$-module $D$ such that the maps
$$
{\rm{id}} \otimes {\varphi_{\alpha}}:\varphi _{\alpha}^\ast D: = E_{L,\Delta} \mathop \bigotimes \limits_{E_{L,\Delta}, \varphi _{\alpha}} D \longrightarrow D,
$$
are isomorphisms for all $\alpha\in\Delta$. 

We are going to show that ${\rm Rep}_{\Bbb{F}_p}(\mathcal{G}_{L,\Delta})$ is equivalent to the category of \' etale 
$(\varphi_\Delta, G_{L,\Delta})$-modules over $E_{L,\Delta}$; the later category we denote by 
$\mathcal{D}^{\rm et}(\varphi_\Delta, G_{L,\Delta}, E_{L,\Delta})$.

\section{Integrality Properties}\label{xxsec4}

\subsection{Definition and projectivity}\label{xxsce4.1}

In this Section our goal is to show that any object in the category  $\mathcal{D}^{\rm et}(\varphi_\Delta, G_{L,\Delta}, E_{L,\Delta})$ is stably free as a module over $E_{L,\Delta}$.

\begin{lemma}\label{Lem4.1}
There exists a $G_{L,\Delta}$-equivariant injective resolution of $E_{L,\Delta}$ as a module over itself.
\end{lemma}

\begin{proof}
This follows from a general result that the Cousin complex provides an injective resolution for 
spectrum of noetherian rings with finite injective dimension (these are the so-called Gorenstein rings,  
we recommend the reader to refer to \cite[Remark before Proposition 3.4 in Page 249]{Hartshorne} or \cite{Sharp}). Note that $E_{L,\Delta}$ is the localization of a $|\Delta|$-variable power series ring over the regular ring $k_\Delta$, therefore it is regular (and in particular, Gorenstein).
Hence the Cousin complex 
$$
0 \xrightarrow[]{\hspace{12pt}}  {E_{L,\Delta} } \xrightarrow[d_{-1}]{\hspace{12pt}} \mathop \bigoplus_{\mathfrak{p} \in {\rm{Spec}}({E_{L,\Delta} }), {\rm ht}\mathfrak{p}=0}{(E_{L,\Delta})_{\mathfrak{p}}} \xrightarrow[d_0]{\hspace{12pt}}   \cdots  \xrightarrow[]{\hspace{12pt}}  \mathop  \bigoplus_{\mathfrak{p} \in {\rm{Spec}}({E_{L,\Delta} }), {\rm ht}\mathfrak{p}=r}{{\rm Coker}(d_{r-2})_{\mathfrak{p}}} \xrightarrow[d_r]{\hspace{12pt}}   \cdots 
$$
is an injective resolution of $E_{L,\Delta}$. This resolution is $G_{L,\Delta}$-equivariant because the automorphisms preserve the height of a prime ideal.  
\end{proof}

Recall $\Bbb{F}_\Delta$ has primitive idempotents $b_1,\dots,b_\ell\in \Bbb{F}_\Delta$ with $1=b_1+\cdots+b_\ell$ and $b_j\Bbb{F}_\Delta\cong \Bbb{F}_{p^f}$ where $f=\gcd(|\Bbb{F}_\alpha:\Bbb{F}_p| \mid \alpha\in\Delta)$, $1\leq j\leq \ell$.

\begin{lemma}\label{Lem4.2}
Let $I$ be a $G_{L,\Delta}$-invariant ideal of $E_{L,\Delta}$. Then we have $I=(I\cap \Bbb{F}_\Delta)E_{L,\Delta}$. 
\end{lemma}

\begin{proof}
Suppose that $I$ is a nontrivial $G_{L,\Delta}$-invariant ideal of $E_{L,\Delta}$. Then $I$ is also 
$\Gamma_\Delta$-invariant. We first show $I=(I\cap k_\Delta )E_{L,\Delta}$. This is completely analogous to the proof of \cite[Proposition 2.1, Lemma 2.2]{Zabradi1}. 
The assumption that the ring $\kappa$ of loc. cit. is a finite field is not used in the proof of \cite[Lemma 2.2]{Zabradi1}: one only 
needs that $E_{L,\Delta}$ is noetherian. Further, $t$ can be chosen large enough to ensure $1+p^t$ lies in the image of 
the character $\chi\colon \Gamma_{L_\alpha}\to\Bbb{Z}_p^\times$ for all $\alpha\in\Delta$. In our argument $\kappa$ is not 
a field, but the ring $k_\Delta$ in which case one cannot conclude at the end of the proof of Proposition 2.1 that $1\in I$ but only that $I=(I\cap k_\Delta )E_{L,\Delta}$.

We use the action of $H_{L,\Delta}$ in order to further descend to $\Bbb{F}_\Delta$. Fix $\alpha\in\Delta$, $j\in \{1,\dots,\ell\}$, and pick an element $0\neq\lambda=\sum_{i=1}^ru_i\otimes v_i\in b_jI\cap k_\Delta$ where $u_1,\dots,u_r\in k_\alpha$ and $v_1,\dots,v_r\in k_{\Delta\setminus\{\alpha\}}$. Suppose $r$ is minimal and $u_1=1$ (possibly replacing $\lambda$ with $u_1^{-1}\lambda$), i.e.\ no nonzero element in $b_jI\cap k_\Delta$ can be written as a sum of at most $r-1$ elementary tensors. In particular, both $u_1,\dots,u_r$ and $v_1,\dots,v_r$ linearly independent over $\Bbb{F}_p$. Assume for contradiction that there is an index $i_0\in\{2,\dots,r\}$ such that $u_{i_0}\notin \Bbb{F}_\alpha$. Then there exists an element $h\in H_{L_\alpha}$ such that $h(u_{i_0})\neq u_{i_0}$ whence $0\neq \sum_{i=2}^r (h(u_i)-u_i)\otimes v_i = h(\lambda)-\lambda\in b_jI$. Writing $h(\lambda)-\lambda=\sum_{N=0}^\infty\lambda_NX_\alpha^N$ where $\lambda_N\in b_jI\cap k_\Delta$ can be written as a sum of at most $r-1$ elementary tensors for all $N\geq 0$ contradicts to the minimality of $r$. By repreating this argument for all $\alpha\in\Delta$ we find a nonzero element in $b_jI\cap \Bbb{F}_\Delta$ whenever $b_jI\neq 0$. Hence we deduce $b_j I= b_jE_{L,\Delta}$ as $b_j\Bbb{F}_\Delta$ is a field. The claim follows noting that $j$ is arbitrary and $I=\bigoplus_{j=1}^\ell b_jI$.
\end{proof}

\begin{lemma}\label{Lem4.3}
Any object $D$ in $\mathcal{D}^{\rm et}(\varphi_\Delta, G_{L,\Delta}, E_{L,\Delta})$ is a projective module over $E_{L,\Delta}$.
\end{lemma}

\begin{proof}
The proof follows from the argument in \cite[Proposition 2.2]{Zabradi2} with the following modification: since $E_{L,\Delta}$ is not a domain, the assertion that ${\rm Ext}^n_{E_{L,\Delta}}(D,E_{L,\Delta})=0$ does not make sense. However, ${\rm Ext}^n_{E_{L,\Delta}}(D,E_{L,\Delta})=\bigoplus _{j=1}^\ell {\rm Ext}^n_{b_jE_{L,\Delta}}(b_jD,b_jE_{L,\Delta})$ is a finitely generated torsion $b_jE_{L,\Delta}$-module ($b_j E_{L,\Delta}$ is a domain). So we deduce by the proof of \cite[Proposition 2.2]{Zabradi2} that $b_jD$ is a projective $E_{L,\Delta}$-module for each $j=1,\dots,\ell$ whence so is $D=\bigoplus_{j=1}^\ell b_jD$.
\end{proof}

\begin{lemma}\label{Lem4.4}
We have ${K_0}(b_j{E_{L,\Delta} }) \cong \Bbb{Z}$ for all $j=1,\dots,\ell$, ie. any finitely generated projective module over $b_j{E_{L,\Delta} }$ is
stably free.
\end{lemma}

\begin{proof}
The proof was given in \cite[Lemma 2.3]{Zabradi2}. 
\end{proof}

\begin{proposition}
Let $D$ be an object in the category  $\mathcal{D}^{\rm et}(\varphi_\Delta, G_{L,\Delta}, E_{L,\Delta})$. Then $D$ is stably free as a module over $E_{L,\Delta}$.
\end{proposition}

\begin{proof}
By Lemma \ref{Lem4.4} it remains to show that the rank of $b_jD$ does not depend on $j$ ($j=1,\dots,\ell$). However, this follows from Lemma \ref{Lem2.4} as we have $E_{L,\Delta}\varphi_\alpha(b_jD)=\varphi_\alpha(b_j)D$ for all $\alpha\in\Delta$.
\end{proof}

\vspace{3mm}

\subsection{Topology of $E_{L,\Delta}^+$ and $E_{L,\Delta}$}\label{xxsec4.2}

We equip $E_{L,\Delta}^+$ with the $X_\Delta $-adic topology, and equip $E_{L,\Delta}$ with the inductive 
limit topology $E_{L,\Delta}=\bigcup \limits_n X_\Delta^{-n}E_{L,\Delta} ^+$. This makes 
$(E_{L,\Delta},E_{L,\Delta} ^+)$ a Huber pair in the sense
of \cite{ScholzeWeinstein}. $E_{L,\Delta}$ is a complete noetherian Tate ring (loc. cit.). Note that this is not the natural compact topology
on $E_{L,\Delta}^+$ as in the compact topology $E_{L,\Delta}^+$ would not be open in ${E_{L,\Delta} }$ since the index of $E_{L,\Delta} ^ + $ in
${X_\Delta ^{ - n}E_{L,\Delta} ^ + }$ is not finite. Also, the inclusion $k_\alpha(({X_\alpha })) \hookrightarrow E_{L,\Delta}$ is not continuous in
the ${X_\Delta }$-adic topology (unless $|\Delta|=1$).

Suppose $D\in {\mathcal{D}^{\rm et}}(\varphi _\Delta, G_{L,\Delta}, E_{L,\Delta})$. 
By Banach's Theorem for Tate rings (\cite[Proposition 6.18]{Wedhorn}), 
there is a unique ${E_{L,\Delta} }$-module topology on $D$ that we call the ${X_\Delta }$-\textit{adic topology}. 
(This is the induced topology as $D$ is finitely generated over $E_{L,\Delta}$). 
Moreover, any $E_{L,\Delta}$-module homomorphism is continuous in the $X_\Delta$-adic topology.

Let $M$ be a finitely generated $E_{L,\Delta}^+$-submodule in $D\in \mathcal{D}^{\rm et}(\varphi_\Delta, G_{L,\Delta}, E_{L,\Delta})$. 
Suppose that $\{m_1, m_2, \cdots, m_n\}$ is a set of generators of $M$. Then 
${\varphi _s}({m_1}), \cdots ,{\varphi_s}({m_n})$ generate $E_{L,\Delta}^+ {\varphi_s}(M)$.
Thus $E_{L,\Delta}^+ {\varphi_s}(M)$ is also finitely generated.

Now, let $D^{+ +}:=\left\{ x \in D\,  \vline\,  \lim \limits_{k \to \infty } \varphi_s^k(x) = 0 \right\}$ where the
limit is considered in the ${X_\Delta }$-adic topology (cf. \cite[II.2.1]{Colmez1} 
in case $|\Delta | = 1$).

\begin{proposition}\label{Prop4.5}
${D^{ +  + }}$ is a finitely generated $E_{L,\Delta} ^ + $-submodule in $D$ which is stable under the
actions of $\varphi_\alpha, G_{L_\alpha}$ for all $\alpha\in \Delta$ and we have $D=D^{++}[X_\Delta ^{-1}]$.
\end{proposition}

\begin{proof}
The proof is essentially the same as \cite[Proposition 2.5]{Zabradi2}, but we would like to 
clarify few steps as we sketch the third author's line of proof. 

Choose an arbitrary finitely generated $E_{L,\Delta}^+$-submodule $M$ of $D$ with $M[X_\Delta ^{-1}]=D$.
We can take $M=E_{L,\Delta} ^ + {e_1}+ \cdots  + E_{L,\Delta}^+{e_n}$ for some ${E_{L,\Delta} }$-generating 
system ${e_1}, \cdots ,{e_n}$ of $D$. First, note that $M$ is not $\varphi_s$-stable, but 
$E^+_\Delta\varphi_s(M)$ is finitely generated (as $M$ is finitely generated over $E^+_\Delta$). Hence we can find 
a ``common denominator" of $E_{L,\Delta}^+\varphi_s(M)$ to be $X_\Delta^r$ such that ${\varphi _s}(M) \subseteq X_\Delta ^{ - r}M$, 
since $E_{L,\Delta} ^ + $ is noetherian and we have
$D = \bigcup \limits_r {X_\Delta ^{ - r}M} $. Then we have
$${\varphi _s}(X_\Delta ^kM) = X_\Delta ^{pk}{\varphi _s}(M) \subseteq X_\Delta ^{pk - r}M \subseteq X_\Delta ^{k + 1}M$$
for any integer $k \ge {{r + 1} \over {p - 1}}$. We therefore have $X_\Delta ^{[{{r + 1} \over {p- 1}}]{+1}}M \subseteq {D^{{\rm{ +  + }}}}$. 
This implies that 
$$
M[X_\Delta^{-1}]=D=X_\Delta^{[{{r+1} \over {p-1}}]{+1}}M[X_\Delta^{-1}] \subseteq D^{++}[X_\Delta^{-1}].
$$
But $D^{++}[X_\Delta^{-1}]\subseteq D$ is obvious. Thus $D^{++}[X_\Delta^{-1}]=D$. Note that $D^{+ +}$ is stable
under $G_{L_\alpha}$, because the action of $G_{L_\alpha}$ commute with $\varphi_s$ (and also $\varphi_\alpha$ for 
all $\alpha\in \Delta$). There is a system of neighbourhoods of 0 in $D$ consisting of $E_{L,\Delta} ^ + $-submodules. 
And hence $D^{++}$ is an $E_{L,\Delta} ^+$-submodule.

Assume first $D$ is a free module over ${E_{L,\Delta} }$ with free generators ${e_1}, \cdots ,{e_n}$ and put $M:= E_{L,\Delta} ^ + {e_1} +  \cdots  + E_{L,\Delta} ^ + {e_n}$. 
Then we can show that $D^{++} \subseteq X_\Delta ^{- r}M$ for some integer $r>0$ (cf. \cite[Proposition 2.5]{Zabradi2}). 
As $E_{L,\Delta}^+$ is noetherian, this gives that $D^{++}$ is finitely generated over $E_{L,\Delta}^+$. 

In the general case, by Proposition \ref{Prop4.5}, we know that $D$ is stably free. Therefore we can have 
${D_1}: = D \bigoplus E_{L,\Delta} ^k$ making $D_1$ into an \' etale free module over $(\varphi_\alpha, \varphi_s, G_{L_\alpha}, \alpha\in \Delta)$
by the trivial action of $(\varphi_\alpha, \varphi_s, G_{L_\alpha}, \alpha\in \Delta)$ on $E_{L,\Delta}^k$. 
This gives us that $D_1^{+ +}$ is finitely generated over
$E_{L,\Delta} ^+$. The result follows as ${D^{ +  + }} \subseteq D_1^{ +  + }$
and $E_{L,\Delta} ^ + $ is noetherian.
\end{proof}

Let us define $${D^ + }: = \{ x \in D\, |\, \{ \varphi _s^k(x):k \ge 0\}  \subset D{\rm{ \ is \ bounded}}\} .$$
Since $\varphi _s^k(X_\Delta)$ tends to 0 in the ${X_\Delta }$-adic topology, we have 
$X_\Delta {D^ + } \subseteq D^{++}$, i.e. $D^+\subseteq X_\Delta ^{ - 1}{D^{++}}$.
In particular, $D^+$ is finitely generated over $E_{L,\Delta} ^ + $. On the other hand, we also have ${D^{ +  + }} \subseteq {D^ + }$ by construction whence we deduce $D = {D^ + }[X_\Delta ^{ - 1}]$.

\begin{lemma}\label{Lem4.6}
For all $\alpha\in\Delta$ and $g_\alpha\in G_{L_\alpha}$ we have 
$$
\begin{aligned}
\varphi _\alpha(D^+) \subset D^+ & \, \, \, \, \, ({\rm resp}. \, \, \varphi_\alpha(D^{++}) \subset D^{++}),\\
g_{\alpha}(D^+) \subset D^+ &\, \, \, \, \,  ({\rm resp}.\, \, g_{\alpha}(D^{++}) \subset D^{++}).
\end{aligned}
$$
\end{lemma}

\begin{proof}
We will show that, for any generating system $e_1, \cdots, e_n$ of $D$ and any $\gamma$ ($\gamma$ can be $\varphi_\alpha$ or $g_\alpha\in G_{L_\alpha}$), 
there exists an integer $k> 0$ such that 
$$
\gamma(X_\Delta^k M) \subseteq X_\Delta^k E_{L,\Delta}^+ \gamma(M) \subseteq M,
$$
where $M:= E_{L,\Delta} ^ + {e_1} +  \cdots  + E_{L,\Delta} ^ + {e_n}$. 
\vspace{2mm}

{\bf Case (i)} Assume that $\gamma=\varphi_\alpha$. Then choose $k\gg 0$ so that $X_\Delta^kE_{L,\Delta}^+\varphi_\alpha(M)\subseteq M$ whence
$$
\varphi_\alpha(X_\Delta^kM)=\prod_{\beta\neq \alpha} X_\beta^k X_\alpha^{pk} \varphi_\alpha(M)=X_\Delta^k X_\alpha^{(p-1)k} \varphi_\alpha(M)\subseteq M\ .
$$

{\bf Case (ii)} Assume that $\gamma=g_\alpha\in G_{L_\alpha}$. We need to check that $g_\alpha(X_\Delta)=uX_\Delta$ for some unit $u\in E_{L,\Delta}^+$. In case $K=L$ this is clear from the intrinsic description of the action of $G_{L_\alpha}$. The general statement follows noting that we still have $g_\alpha(X_\beta)=X_\beta$ for $\beta\neq\alpha$ and $g_\alpha(X_\alpha)$ is also a uniformizer in $k_\alpha((X_\alpha))$ since $k_\alpha((X_\alpha))$ is a finite separable extension of $k_\alpha((\overline{\pi}_\alpha))$ and $G_{L_\alpha}$ is a subgroup in $G_{K_\alpha}$.
The proof then follows from \cite[Lemma 2.6]{Zabradi2}. 
\end{proof}

We now fix an $\alpha\in \Delta$ and define
$D^+_{\overline{\alpha}}:=D^+[X^{-1}_{\Delta\backslash \{\alpha\} }]$
where for any subset $S \subseteq \Delta$ we put
${X_S}: = \prod_{\beta  \in S} {{X_\beta }} $. Then $D_{\overline \alpha}^ +$
is a finitely generated module over $E_{\overline \alpha}^+:=E_{L,\Delta}^+[X_{\Delta \backslash \{\alpha\} }^{-1}]$.

\begin{lemma}\label{Lem4.7}
$D_{\overline \alpha}^+/{D^ + }$ is ${X_\alpha }$-torsion free: If both $X_\alpha ^{{n_1}}d$ and $X_{\Delta \backslash \{ \alpha \} }^{{n_2}}d$ lie in $D^+$ for some
element $d \in D$, $\alpha  \in \Delta $, and integers $n_1, n_2 \ge 0$ then we have $d \in {D^ + }$. The same statement
holds if we replace ${D^+}$ by $D^{++}$.
\end{lemma}

\begin{proof}
The proof of \cite[Lemma 2.7]{Zabradi2} works without any change. 
\end{proof}



\begin{lemma}\label{Lem4.8}
Assume that $D$ is generated by a single element ${e_1} \in D$ over $E_{L,\Delta}$. Then for any
$\gamma$, we have $\gamma(e_1) = {a_\gamma}e_1$ for some unit $a_\gamma\in {(E_{\overline \alpha  }^ + )^ \times }$.
Here $\gamma$ can be $\varphi_\beta, g_\beta\in G_{L_\beta}$ for $\beta\neq \alpha$. 
\end{lemma}

\begin{proof}
For any $\gamma$ equals to either $g_\beta$ or $\varphi_\beta$, we define $a_\gamma$ and $a_\alpha$ such that
$$
\gamma(e_1)=a_\gamma e_1\, \, \text{and}\, \, \varphi_\alpha(e_1)=a_\alpha e_1.
$$ 
By the \'etaleness property, it follows that $D$ is generated by $\gamma(D)$ over $E_{L,\Delta}$.
Thus $e_1\in D$ implies 
$$
\begin{aligned}
e_1 & = x \gamma(e_1)\, \, \, \, (\text{for some} \, \, x\in E_{L,\Delta})\\
&= x a_\gamma e_1\\
& = y \varphi_\alpha(e_1)\, \, \, \, (\text{for some}\, \, y\in E_{L,\Delta}, \text{as} \, \, D \,\,  \text{is also generated by} \, \, \varphi_\alpha(D) \, \, \text{over} \, \, E_{L,\Delta})\\
&= y a_\alpha e_1
\end{aligned}
$$
So $x a_\gamma=ya_\alpha=1$, which implies that both $a_\gamma$ and $a_\alpha$ are units in $E_{L,\Delta}$.
It remains to show that ${\rm val}_{X_\alpha}(a_\gamma) = 0$. We
compute
$$
\varphi _\alpha(a_\gamma)a_\alpha e_1= \varphi _\alpha(a_\gamma)\varphi _\alpha (e_1) = \varphi_\alpha (a_\gamma e_1) = \varphi _\alpha(\gamma(e_1)) 
$$
$$ 
= \gamma (\varphi _\alpha(e_1)) = \gamma(a_\alpha  e_1) = \gamma (a_\alpha) \gamma(e_1) = \gamma (a_\alpha )a_\gamma e_1.
$$
And hence we deduce 
$$
p {\rm val}_{X_\alpha}(a_\gamma) + {\rm val}_{X_\alpha}(a_\alpha) = {\rm val}_{X_\alpha}(\varphi_\alpha (a_\gamma) a_\alpha ) = {\rm val}_{X_\alpha}(\gamma (a_\alpha) a_\gamma).\eqno{(4.8.1)}
$$
Since $\varphi_\beta$ and $g_\beta$ act trivially on $k_\alpha((X_\alpha))$ and they are injective on $E_{L,\Delta}$, they both preserve the $X_\alpha$-adic valuation. So we have 
$$
{\rm val}_{X_\alpha}(\gamma (a_\alpha) a_\gamma) = {\rm  val}_{X_\alpha}(a_\alpha ) + {\rm val}_{X_\alpha}(a_\gamma) \eqno{(4.8.2)}
$$
for $\gamma=\varphi_\beta, g_\beta$ where $\beta\neq \alpha$. Hence, by (4.8.1) we obtain 
$$
p {\rm val}_{X_\alpha}(a_\gamma) + {\rm val}_{X_\alpha}(a_\alpha) = {\rm  val}_{X_\alpha}(a_\alpha ) + {\rm val}_{X_\alpha}(a_\gamma).\eqno{(4.8.3)}
$$
Now (4.8.3) immediately yields ${\rm val}_{X_\alpha}(a_\gamma)=0$ as desired.
\end{proof}

\begin{lemma}\label{Lem4.9}
There exists an integer $k = k(D) > 0$ such that for any $\gamma \in \{\varphi_\beta\mid \beta\in \Delta\setminus\{\alpha\}\}\cup G_{L,\Delta}$, we have
$$
X_\alpha ^kD_{\overline{\alpha}}^ +\subseteq E_{L,\Delta}^+\gamma(D_{\overline{\alpha}}^+)\subseteq E^+_{\overline{\alpha}} \gamma(D^+_{\overline{\alpha}}).
$$
\end{lemma}

\begin{proof}
The proof for the first inclusion relation follows exactly as in \cite[Lemma 2.9]{Zabradi2}. The 
second inclusion relation is obvious as $E_{L,\Delta}^+\subseteq E^+_{\overline{\alpha}}$ by definition. 
\end{proof}

Let us now define
$$
D_{\overline \alpha  }^{ + *}: = \bigcap_{\gamma} E_{\overline \alpha  }^ + \gamma (D_{\overline{\alpha}}^ +),
$$
where $\gamma$ runs on the operators $\varphi_\beta$ for $\beta\neq \alpha$ and on $G_{L,\Delta}$.  $D_{\overline \alpha  }^{+ \ast}$
is finitely generated over $E_{\overline \alpha}^ +$ as it is contained in ${D_{\overline \alpha  }^ + }$
and ${E_{\overline \alpha  }^ + }$ is noetherian. By Lemma \ref{Lem4.9}, we conclude that  
$X_\alpha ^kD_{\overline \alpha }^+ \subseteq D_{\overline \alpha  }^{+ \ast}$
for some integer $k = k(D) > 0$. In particular, $D = D_{\overline \alpha  }^{+ \ast}[X_\alpha ^{ - 1}]$.

\begin{proposition}\label{Prop4.10}
$D_{\overline \alpha  }^{ + \ast}$ is an \' etale module over $E_{\overline \alpha  }^+$ , i.e. the maps
$$
{\rm id} \otimes \gamma \colon \gamma^\ast D_{\overline \alpha  }^{ + *} = E_{\overline \alpha}^ + \bigotimes \limits_{E_{\overline \alpha }^ +, \gamma} D_{\overline \alpha }^{ + \ast} \longrightarrow  D_{\overline \alpha  }^{+ \ast}
$$
are bijective for all $\gamma \in  \{\varphi_\beta\mid \beta\in \Delta\setminus\{\alpha\}\}\cup G_{L,\Delta}$.
\end{proposition}

\begin{proof}
The only thing we need to check in order for the third author's arguments in \cite[Proposition 2.10]{Zabradi2} 
to work is that $E_{\overline \alpha  }^ + $ (resp. $E_{L,\Delta}$, resp. $E^+_\Delta$) is a finite free module over $\gamma(E_{\overline \alpha}^+)$ 
(resp. over $\gamma(E_{L,\Delta})$, resp. over $\gamma(E_{L,\Delta}^+)$). This is already true if $\gamma=\varphi_\beta$ 
because the action of $\varphi_\beta$ on the variables is exactly the same as the third author's arguments. 
If $\gamma\in G_{L,\Delta}$ then this is automatic since $G_{L,\Delta}$ is a group, so the action of $\gamma$ is bijective.
The rest of the argument in the proof follows exactly as in \cite[Proposition 2.10]{Zabradi2}. 
\end{proof}

\begin{lemma}\label{Lem4.12}
There exists a finitely generated $E_{L,\Delta} ^+$-submodule $D_0\subset D_{\overline \alpha}^{+ \ast}$
such that $D_0\subseteq E_{L,\Delta}^+ \varphi _{\overline \alpha}(D_0)$ and $D_{\overline \alpha}^{+ \ast}=D_0 [X_{\Delta \backslash \{ \alpha \} }^{-1}]$, 
where $\varphi _{\overline \alpha}:= \prod \limits_{\beta  \in \Delta \backslash \{ \alpha \}} {\varphi _\beta} $. 
Moreover, we have 
$$
D_{\overline \alpha}^{+ \ast}=\bigcup_{r\ge 0} E_{L,\Delta} ^+ \varphi _{\overline \alpha}^r (X_{\Delta \backslash \{ \alpha \} }^{ - 1}D_0).
$$
\end{lemma}

\begin{proof}
The proof is exactly the same as \cite[Lemma 2.11]{Zabradi2}. 
\end{proof}

\section{ The equivalence of categories for $\Bbb{F}_p$-representations}\label{xxsec5}

\subsection{The functor $\Bbb{D}$}\label{ccsec5.1}
Let  $\mathcal{H}_{L,\Delta}:=\prod \limits_{\alpha\in \Delta}\mathcal{H}_{L_\alpha}$, where $\mathcal{H}_{L_\alpha} := {\rm Gal}({\overline K_\alpha} /L_{\alpha,\infty})$ for each $\alpha\in\Delta$. In case $L_\alpha=K_\alpha$ we omit the subscript $K$ from the notation, ie.\ we put $\mathcal{H}_\Delta:=\prod \limits_{\alpha\in \Delta}\mathcal{H}_\alpha$ where 
$\mathcal{H}_\alpha := {\rm Gal}({\overline K_\alpha} /K_{\alpha,\infty})$. Recall the field of norm $E_\alpha$ of $K_{\alpha,\infty}$ 
is isomorphic to $E_\alpha\cong k_\alpha ((\overline{\pi}_\alpha ))$. We already 
know by \cite[Corollary 6.4]{Andreatta} that $(E_\alpha^{\rm sep})^{\mathcal{H}_\alpha}\cong E_\alpha$. 
For each $\alpha\in \Delta$, consider a finite separable extension $E^\prime_\alpha$ of $E_\alpha$ 
together with the natural Frobenius $\varphi_\alpha: E^\prime_\alpha\longrightarrow E^\prime_\alpha$. 
The structure theorem for local fields of equal characteristic shows that 
$E^\prime_\alpha \cong k_\alpha ((X_\alpha^\prime ))$, where $k_\alpha$ is a finite 
extension over $k_\alpha$. The field $k_\alpha$ is also the residue field of $E^\prime_\alpha$ and $X_\alpha^\prime$ 
is a uniformizer of $E_\alpha^\prime$. We denote by $E^{\prime + }_\alpha\cong k_\alpha \llbracket X_\alpha^\prime \rrbracket$ 
in $E^\prime_\alpha$. As in \cite[Section 3.1]{Zabradi2}, we equip the tensor product 
$$
E^\prime_{\Delta , \circ}: = \mathop  \bigotimes \limits_{\alpha  \in \Delta, \Bbb{F}_p} E^\prime_\alpha 
$$ 
with a norm $| \cdot |_{\rm prod}$ by the
formula
$$
| c |_{\rm prod}: = \inf \left( {\mathop {\max }\limits_i \left( \prod\limits_{\alpha  \in \Delta } | c_{\alpha ,i}  |_\alpha  \right)\,  \vline \, c  = \sum\limits_{i = 1}^n {\mathop  \bigotimes \limits_{\alpha  \in \Delta } c_{\alpha ,i} }} \right).
$$
Note that the restriction of $| \cdot |_{\rm prod}$ to the subring $E^{\prime + }_{\Delta , \circ }:=\mathop  \bigotimes \limits_{\alpha  \in \Delta, \Bbb{F}_p} E^{\prime + }_\alpha$
induces the valuation with respect to the augmentation ideal ${\rm Ker}(E^{\prime + }_{\Delta , \circ } \longtwoheadrightarrow \mathop  \bigotimes \limits_{\alpha  \in \Delta , \Bbb{F}_p} k_\alpha)$. 
Note that $\mathop  \bigotimes \limits_{\alpha  \in \Delta, k_\alpha} k_\alpha$ is not a domain, and hence $| \cdot |_{\rm prod}$
is not multiplicative in general. However, it is
submultiplicative. Following \cite{Zabradi2}, we define $E^{\prime + }_\Delta$ as the completion of 
$E^{\prime + }_{\Delta , \circ }$ with respect to $| \cdot |_{\rm prod}$ and put
$E'_\Delta: = E^{\prime + }_\Delta  [1/{X_\Delta }]$. This ring $E^\prime_\Delta$ is not complete 
with respect to $| \cdot |_{\rm prod}$ (unless $ |\Delta | = 1$).
 Further, ${\varphi _\alpha }$ acts on $E^{\prime + }_{\Delta , \circ}$
(and on $E^{\prime}_{\Delta , \circ}$) by the Frobenius on the component $E^\prime_\alpha$ in $E^\prime_\Delta$
and by the identity on all the other components in $E^\prime_\beta$ for $\beta  \in \Delta \backslash \{ \alpha \} $. 
This action is continuous in the norm $| \cdot |_{\rm prod}$ and therefore extends to the completion
$E^{\prime +}_\Delta$ and the localization $E^{\prime}_\Delta$.

We define the multivariable analogue of $E^{\rm sep}$ as
$$
E_\Delta^{\rm sep}: = \mathop {\underrightarrow {\lim }}\limits_{{E_\alpha } \leqslant E_\alpha ^\prime \leqslant E_\alpha ^{sep},\forall \alpha  \in \Delta } E_\Delta ^\prime.
$$
For any subset $\Delta^\prime \subseteq \Delta$,  one can define the similar notions $E_{\Delta^\prime}^{\prime + }$, 
$E_{\Delta^\prime}^{\prime}$ and $E_{\Delta^\prime}^{\rm sep}$ with $\Delta$ replaced
by $\Delta^\prime$. We equip $E_\Delta ^{\rm sep}$ with the relative Frobenii ${\varphi _\alpha }$ 
for each ${\alpha  \in \Delta }$ and the absolute Frobenius $\varphi_s$ defined above on 
each $E^\prime_\Delta$. Further $E_\Delta ^{\rm sep}$ admits a Galois action of the Galois group 
$\mathcal{G}_\Delta$.

\vspace{2mm}

With respect to the ring $E^{\prime}_\Delta$, we have the following alternative characterization.

\begin{lemma}\label{Lem5.1}
Put $\Delta  = \{ {\alpha _1}, \cdots ,{\alpha _n}\} $. We have
$$
E_\Delta ^\prime \cong E_{{\alpha _1}}^\prime{ \bigotimes \limits_{{E_{{\alpha _1}}}}}\left(E_{{\alpha _2}}^\prime{ \bigotimes \limits_{{E_{{\alpha _2}}}}}\left( \cdots \left(E_{{\alpha _n}}^\prime{ \bigotimes \limits_{{E_{{\alpha _n}}}}}{E_\Delta }\right)\right)\right).
$$

\end{lemma}

\begin{proof}
The proof of \cite[Lemma 3.2]{Zabradi2} also works in our imperfect case.
\end{proof}

\begin{proposition}\label{Prop5.2}
For a collection of finite separable extensions $E_\alpha ^\prime/E_\alpha$ ${\alpha  \in \Delta }$ we let
$\mathcal{H}_\Delta^\prime: = \prod\limits_{\alpha  \in \Delta } {\mathcal{H}_\alpha ^\prime} $, 
where $\mathcal{H}_\alpha^\prime = {\rm Gal}(E_\alpha ^{\rm sep}/E_\alpha ^\prime)$. 
Then we have $(E_\Delta ^{\rm sep})^{\mathcal{H}_\Delta ^\prime} = E_\Delta ^\prime$. In particular,
we have $(E_\Delta ^{\rm sep})^{\mathcal{H}_{L,\Delta}}=E_{L,\Delta}$.
\end{proposition}

\begin{proof}
Since $X_\Delta$ is $\mathcal{H}_\Delta ^\prime$-invariant and $\underrightarrow {\lim }$ 
can be interchanged with taking $\mathcal{H}_\Delta ^\prime$-invariants, it
suffices to show that whenever
$$
E_\alpha= k_\alpha(({X_\alpha })) \leqslant E_\alpha ^\prime = k_{\alpha^\prime}((X_\alpha ^\prime)) \leqslant E_\alpha ^{\prime\prime} = k_{\alpha^{\prime\prime}}((X_\alpha ^{\prime\prime}))
$$
is a sequence of finite separable extensions for each ${\alpha  \in \Delta }$ such that $E_\alpha ^{\prime\prime}/E_\alpha ^\prime$ is Galois, then we have $(E_\Delta ^{\prime\prime + })^{\mathcal{H}_\Delta ^\prime}= E_\Delta ^{\prime + }$. 
The containment $ E_\Delta ^{\prime + } \subseteq (E_\Delta ^{\prime\prime + })^{\mathcal{H}_\Delta ^\prime}$ 
is clear. For the converse, we will prove by induction on $|\Delta |$. It should be remarked that 
the ideal $\mathcal{M}_\alpha \triangleleft E_\Delta ^{\prime\prime + }$ generated by $X_\alpha ^{\prime\prime}$
is invariant under the action of $\mathcal{H}_\Delta ^\prime$ for any fixed $\alpha\in \Delta$. 
Moreover, for any integer $k \geq 1$,  the ring $E_\alpha ^{\prime\prime + }/\mathcal{M}_\alpha ^k$
is finite dimensional over $k_\alpha$. Therefore
the image of $(E_\Delta ^{\prime\prime + })^{\mathcal{H}_\Delta ^\prime}$ under the quotient map $E_\Delta ^{\prime\prime + } \longtwoheadrightarrow E_\Delta ^{\prime\prime + }/\mathcal{M}_\alpha ^k$
is contained in
$$
\begin{aligned}
\left( E_\Delta ^{\prime\prime + }/\mathcal{M}_\alpha ^k \right)^{\mathcal{H}_\Delta ^\prime}   \subseteq \left( E_\Delta ^{\prime\prime + }/\mathcal{M}_\alpha ^k \right)^{\mathcal{H}_{\Delta \backslash \{ \alpha \} }^\prime} & = {\left(E_{\Delta \backslash \{ \alpha \} }^{\prime\prime + }{ \bigotimes \limits_{\Bbb{F}_p}}\left( E_\alpha ^{\prime\prime + }/\mathcal{M}_\alpha ^k \right) \right)^{\mathcal{H}_{\Delta \backslash \{ \alpha \} }^\prime}} \\
 & = {\left(E_{\Delta \backslash \{ \alpha \} }^{\prime\prime + }\right)^{\mathcal{H}_{\Delta \backslash \{ \alpha \} }^\prime}}{ \bigotimes \limits_{\Bbb{F}_p}}(E_\alpha ^{\prime\prime + }/\mathcal{M}_\alpha ^k) \\
 & = E_{\Delta \backslash \{ \alpha \} }^{\prime + }{ \bigotimes \limits_{\Bbb{F}_p}}(E_\alpha ^{\prime\prime + }/\mathcal{M}_\alpha ^k)&
\end{aligned}\eqno{(5.2.1)}
$$
by induction. Note that the second equality in (5.2.1) follows from the following fact.

{\bf Fact}. If $A$ and $B$ are $k$-vector spaces with a $G$-action such that $B$ is finite dimensional over $k$ 
and $B$ has trivial $G$-action, then  
$$
(A\bigotimes \limits_k B)^G\cong (A\bigotimes \limits_k k^r)^G\cong (A^r)^G\cong A^G \bigotimes \limits_k k^r\cong A^G \bigotimes \limits_k B, 
$$
where $B\cong k^r$.

By taking inductive limits of finite dimensional vector spaces and as inductive limit 
commute with these operations, the assumption that $B$ is finite dimensional over $k$ can be removed from 
this fact.

Let us continue our proof. Taking the projective limit of $(E_\Delta^{\prime\prime +}/\mathcal{M}^k_\alpha)^{\mathcal{H}^\prime_{\Delta}}$ 
with respect to $k \ge 1$, we deduce that $(E_\Delta ^{\prime\prime + })^{\mathcal{H}_\Delta ^\prime}$ is contained in the power series ring
$$
\left( k_{\alpha ^{\prime\prime}} \bigotimes \limits_{\Bbb{F}_p}\mathop  \bigotimes \limits_{\beta  \in \Delta \backslash \{ \alpha \} , \Bbb{F}_p} k_{\beta^\prime} \right)\llbracket X_\alpha ^{\prime\prime}, X_\beta ^\prime \, |\,  \beta  \in \Delta \backslash \{ \alpha \} \rrbracket \subseteq E_\Delta ^{\prime\prime + }.
$$
Indeed, even though projective limits do not commute in general with taking ${\mathcal{H}^\prime_{\Delta}}$-invariants, this inclusion is automatic. Now using the action of $\mathcal{H}_\alpha ^\prime$ in a similar argument as above (reducing modulo the $k$-th power
of the ideal generated by all the $X_\beta ^\prime, \beta  \in \Delta \backslash \{ \alpha \}$ for all $k \ge 1$) we deduce the statement.
\end{proof}

We define the subring $E_{\Delta , \circ }^{\rm sep} \cong \mathop  \bigotimes \limits_{\alpha  \in \Delta , \Bbb{F}_p} E_\alpha ^{\rm sep}$
in $E_\Delta ^{\rm sep}$ to be the inductive limits of $E_{\Delta , \circ }^\prime \subseteq E_\Delta ^\prime$ where 
$E_\alpha ^\prime$ runs through the finite separable extensions of $E_\alpha$ for each ${\alpha  \in \Delta }$.

As in section \ref{xxsec2}.\ assume $L_\alpha/K_\alpha$ is a finite totally ramified extension for each $\alpha\in\Delta$, ie.\ $L_\alpha$ also has $k_\alpha$ as residue field. Put $E_{L,\Delta}:=(E_\Delta^{\rm sep})^{\mathcal{H}_{L,\Delta}}$ where $\mathcal{H}_{L,\Delta}:=\prod_{\alpha\in\Delta}\mathcal{H}_{L_\alpha}$. As in the case $L=K$ a $(\varphi_\Delta,G_{L,\Delta})$-module over $E_{L,\Delta}$ is a finitely generated free module $D$ over $E_{L,\Delta}$ together with commuting semilinear actions of $\varphi_\alpha$ and the Galois groups $G_{L_\alpha}$ for all $\alpha\in \Delta$. By an \' etale $(\varphi_\Delta, G_{L,\Delta})$-module over $E_{L,\Delta}$, we mean a
$(\varphi_\Delta, G_{L,\Delta})$-module $D$ such that the maps
$$
{\rm{id}} \otimes {\varphi_{\alpha}}:\varphi _{\alpha}^\ast D: = E_{L,\Delta} \mathop \bigotimes \limits_{E_{L,\Delta}, \varphi _{\alpha}} D \longrightarrow D,
$$
are isomorphisms for all $\alpha\in\Delta$.  

Now let $V$ be a finite dimensional representation of the group $\mathcal{G}_{L,\Delta}$ over $\Bbb{F}_p$. The basechange
$E_\Delta ^{\rm sep}{ \bigotimes \limits_{{\Bbb{F}_p}}}V$ is equipped with the diagonal semilinear action of $\mathcal{G}_{L,\Delta}$ and with the partial Frobenii
$\varphi _\alpha \, (\alpha  \in \Delta)$. These all commute with each other. We define the functor $\Bbb{D}$ as in \cite{Zabradi2} 
$$
{\Bbb D}(V): = (E_\Delta ^{\rm sep}\bigotimes \limits_{\Bbb{F}_p} V)^{\mathcal{H}_{L,\Delta}}.
$$
By Proposition \ref{Prop5.2}, $\Bbb{D}(V)$ is a module over $E_\Delta$ which inherits the action  
of ${\varphi _\alpha }\, (\alpha  \in \Delta )$, and the Galois group $\mathcal{G}_{L,\Delta}$ on $E_\Delta ^{\rm sep}\bigotimes \limits_{{\Bbb{F}_p}} V $. Further, the action of $\mathcal{G}_{L,\Delta}$ factors through its quotient $G_{L,\Delta}=\mathcal{G}_{L,\Delta}/\mathcal{H}_{L,\Delta}$.  We denote the category of \'etale $(\varphi_\Delta,G_{L,\Delta})$-modules over $E_{L,\Delta}$ by $\mathcal{D}^{\rm et}(\varphi_\Delta,G_{L,\Delta},E_{L,\Delta})$. One key Lemma for us is the
following.

\begin{lemma}\label{Lem5.3}
The $E_\Delta ^{\rm sep}$-module $E_\Delta ^{\rm sep} \bigotimes \limits_{{\Bbb{F}_p}}V$ admits a basis consisting of elements fixed by $\mathcal{H}_{L,\Delta}$.
\end{lemma}

\begin{proof}
The same proof given in \cite[Lemma 3.4]{Zabradi2} exactly works here. 
\end{proof}

\begin{lemma}\label{Lem5.4}
We have $(E_\Delta ^{\rm sep})^ \times  \cap E_{L,\Delta} = E_{L,\Delta} ^ \times $.
\end{lemma}

\begin{proof}
Let $u$ be an arbitrary element in $(E_\Delta ^{\rm sep})^ \times  \cap E_{L,\Delta}$. 
Since $u$ is invariant under the action of $\mathcal{H}_{L,\Delta}$, so
is its inverse $u^{-1}$. And hence it also lies in $E_{L,\Delta}$ by Proposition \ref{Prop5.2}.
\end{proof}

\begin{lemma}\label{Lem5.5}
We have $\bigcap\limits_{\alpha  \in \Delta }\left(E_\Delta^{\rm sep}\right)^{{\varphi_\alpha}={\rm id}}  = \Bbb{F}_p$.
\end{lemma}
\begin{proof}

The containment $\Bbb{F}_p \subseteq \bigcap\limits_{\alpha  \in \Delta } \left(E_\Delta ^{\rm sep}\right)^{\varphi _\alpha= {\rm id}}$ is obvious. 
On the other hand,
let $u \in E_\Delta ^{\rm sep}$ be an arbitrary element such that $\varphi _\alpha(u) = u$ for all 
$\alpha  \in \Delta $. We also have $u^p= \varphi _s(u)=u$
as $\varphi _s=\prod\limits_{\alpha\in \Delta} \varphi_\alpha$ is the absolute Frobenius on $E_\Delta ^{\rm sep}$. 
Since $E_\Delta ^{\rm sep}$ is defined to be an inductive limit, $u$ lies
in $E_\Delta ^\prime \cong \left(\mathop  \bigotimes \limits_{\alpha  \in \Delta , \Bbb{F}_p} k_\alpha^\prime \right)\llbracket  X_\alpha ^\prime \, | \, \alpha  \in \Delta \rrbracket 
[X_\Delta ^{ - 1}]$ for some collection $E_\alpha ^\prime = k_\alpha^\prime((X_\alpha ^\prime))\, (\alpha  \in \Delta )$ of
finite separable extensions of $E_\alpha$.

Since $k_\alpha^\prime$ is a finite separable extension of $k_\alpha$, it is a finitely generated field extension of $\Bbb{F}_p$. Therefore by Lemma \ref{lem2.3}  $k_\Delta^\prime:=\mathop  \bigotimes  \limits_{\alpha \in \Delta, \Bbb{F}_p}  k_\alpha^\prime$ is reduced. 
This implies that for each $\alpha\in\Delta$ the norm $|\cdot|_\alpha$ on $E_\Delta^\prime=k_\Delta^\prime \llbracket X_\alpha ^\prime\, |\, \alpha  \in \Delta \rrbracket [X_\Delta^{-1}]$ defined by the $X_\alpha^\prime$-adic valuation is power-multiplicative as the powers of the leading coefficient (with respect to the $X_\alpha^\prime$-degree) of an element cannot vanish. 
Therefore, we have $| u^p |_{\alpha} = | u |_{\alpha}^p$ for all $\alpha\in\Delta$. We deduce $| u |_{\alpha}=1$
unless $u = 0$. In particular, $u$ lies in 
$E_\Delta ^{\prime + } =k_\Delta^\prime \llbracket X_\alpha ^\prime\, |\, \alpha  \in \Delta \rrbracket$. 
The constant term
$u_0\in k_\Delta^\prime$ also satisfies ${\varphi _\alpha }({u_0}) = {u_0}$ for all ${\alpha  \in \Delta }$. 
Now, $ k_\Delta^\prime$ is an infinite dimensional 
vector space over $\Bbb{F}_p$. For a fixed ${\alpha  \in \Delta }$, we can choose elements of 
an $\Bbb{F}_p$-basis ${d_1}, \cdots ,{d_n}$ of $\mathop \bigotimes \limits_{\beta  \in \Delta \backslash \{\alpha\} , \Bbb{F}_p} k_\beta^\prime $
such that ${u_0} = \sum\limits_{i = 1}^n {{c_i} \otimes {d_i}} $ with ${c_i} \in k_\alpha^\prime$. 
This decomposition is unique and we compute
$$
\sum_{i = 1}^n {c_i} \otimes {d_i} = u_0 = \varphi _\alpha (u_0) = \sum_{i = 1}^n {c_i^p \otimes d_i} .
$$
We conclude $c_i = c_i^p$. But, (when $|\Delta|=1$) by \cite[Theorem 2.1.3]{Scholl} and \cite[Equation 2.1.5 and Equation 2.1.6]{Scholl}, 
we know that $\left( E_\alpha^{\rm sep}\right)^{\varphi_\alpha={\rm id}}=\Bbb{F}_p$. 
Therefore ${{c_i} \in \Bbb{F}_p}$ for all $1 \leqslant i \leqslant n$. It follows by induction on $|\Delta |$ that $u_0$ lies in
$\Bbb{F}_p$. Now $u - u_0$ is also fixed by each $\varphi _\alpha\, (\alpha  \in \Delta )$ and $\varphi_s$, but 
we have $| u - u_0 |_{\rm prod} < 1$. This implies by the discussion above that $u = u_0$ is in $\Bbb{F}_p$ as desired.
\end{proof}

\begin{proposition}\label{Prop5.7}
$\Bbb{D}(V)$ is an \' etale $(\varphi_\Delta, G_{L,\Delta}, E_{L,\Delta})$-module over $E_{L,\Delta}$ of rank 
$d: = \dim _{\Bbb{F}_p} V$. Moreover,
we have 
$$
E_\Delta ^{\rm sep} \mathop \bigotimes \limits_{E_{L,\Delta} }\Bbb{D}(V) \cong E_\Delta ^{\rm sep} \mathop \bigotimes \limits_{\Bbb{F}_p} V, 
$$ 
and
$$
V = \bigcap\limits_{\alpha  \in \Delta } \left(E_\Delta ^{\rm sep}\bigotimes _{E_{L,\Delta} } \Bbb{D}(V)\right)^{\varphi _\alpha= {\rm id}} .
$$
\end{proposition}

\begin{proof}
By Proposition \ref{Prop5.2}  and Lemma \ref{Lem5.3}  we can say that $\Bbb{D}(V)$ is a free module of rank $d$ over $E_{L,\Delta}$. 
Moreover, the matrix of $\varphi _\alpha$ in any basis of $\Bbb{D}(V)$ is invertible in $E_\Delta ^{\rm sep}$, 
therefore also in $E_{L,\Delta}$ by Lemma \ref{Lem5.4}. So the action of $(\varphi_\Delta, G_{L,\Delta})$ on $\Bbb{D}(V)$ is \' etale. 
The last statement is a direct consequence of Lemmas \ref{Lem5.3} and \ref{Lem5.5}.
\end{proof}

\begin{lemma}\label{Lem5.8}
For objects $V, V_1, V_2$ in ${\rm Rep}_{\Bbb{F}_p}({\mathcal{G}_{L,\Delta}})$, 
we have $\Bbb{D} (V_1 \bigotimes \limits_{\Bbb{F}_p} V_2) \cong \Bbb{D}(V_1) \bigotimes \limits_{E_{L,\Delta} }\Bbb{D}(V_2)$
and $\Bbb{D}(V^\ast) \cong \Bbb{D}(V)^\ast$.
\end{lemma}

\begin{proof}
The proof of \cite[Lemma 3.8]{Zabradi2} works in our imperfect case. 
\end{proof}

\begin{theorem}\label{Them5.9}
$\Bbb{D}$ is a fully faithful tensor functor from the category ${\rm Rep}_{\Bbb{F}_p}(\mathcal{G}_{L,\Delta})$ to the
category $\mathcal{D}^{\rm et}(\varphi _\Delta, G_{L,\Delta}, E_{L,\Delta})$.
\end{theorem}

\begin{proof}
Let $f:V_1 \longrightarrow  V_2$ be a nonzero morphism in ${\rm Rep}_{\Bbb{F}_p}(\mathcal{G}_{L,\Delta})$. 
Then the $E_\Delta ^{\rm sep}$-linear map
${\rm id} \bigotimes f: E_\Delta ^{\rm sep} \bigotimes \limits_{\Bbb{F}_p} V_1 \longrightarrow E_\Delta ^{\rm sep} \bigotimes \limits_{\Bbb{F}_p} V_2$ 
is also nonzero. By Proposition \ref{Prop5.7} we assert that $\Bbb{D}(f) \ne 0$, and therefore the faithfulness.

Now let $V_1$ and $V_2$ be arbitrary objects in ${\rm Rep}_{\Bbb{F}_p}(\mathcal{G}_{L,\Delta})$ 
and $\theta :\Bbb{D}(V_1) \longrightarrow \Bbb{D}(V_2)$ be a
morphism in $\mathcal{D}^{\rm et}(\varphi _\Delta, G_{L,\Delta}, E_{L,\Delta})$. 
Then by Proposition \ref{Prop5.7}, we obtain a $\mathcal{G}_{L,\Delta}$-equivariant $\Bbb{F}_p$-linear map
$$
f: V_1=\bigcap\limits_{\alpha  \in \Delta} \left(E_\Delta ^{\rm sep} \bigotimes \limits_{E_{L,\Delta}} \Bbb{D}(V_1)\right)^{\varphi _\alpha={\rm id}}  \longrightarrow \bigcap\limits_{\alpha  \in \Delta } \left(E_\Delta ^{\rm sep}\bigotimes \limits_{E_{L,\Delta}}\Bbb{D}(V_2)\right)^{\varphi _\alpha={\rm id}}  = V_2
$$
induced by $\theta$ for which we have $\theta = \Bbb{D}(f)$. Therefore $\Bbb{D}$ is full. The compatibility with tensor
product follows from Lemma \ref{Lem5.8}.
\end{proof}

\subsection{The functor $\Bbb{V}$}\label{xxsec5.2}
In the following, we define the functor $\Bbb{V}$ and show that it is a quasi-inverse of $\Bbb{D}$. This, in turn, implies that the functor $\Bbb{D}$ is essentially surjective. 
Let $D\in \mathcal{D}^{\rm et}(\varphi_\Delta, G_{L,\Delta}, E_{L,\Delta})$. 
It comes with a natural semilinear action of $\varphi_\alpha$ ($\alpha\in\Delta$) and 
the Galois group $G_{L,\Delta}$. We define 
$$
\Bbb{V}(D): = \bigcap\limits_{\alpha  \in \Delta} \left (E_\Delta^{\rm sep} \bigotimes \limits_{E_{L,\Delta}}D \right)^{\varphi_\alpha ={\rm id}} .
$$
$\Bbb{V}(D)$ is a---a priori not necessarily finite dimensional---representation of $\mathcal{G}_{L,\Delta}$ over $\Bbb{F}_p$.

\begin{lemma}\label{Lem5.10}
For any integer $r > 0$ we have 
$$\bigcap\limits_{\beta  \in \Delta \backslash \{\alpha\} } \left(E_{\Delta \backslash \{ \alpha \} }^{\rm sep}\otimes_{E_{L,\Delta\setminus\{\alpha\}}^+}E_{L,\Delta}^+/(X_\alpha ^r) \right)^{\varphi _\beta={\rm id}} =k_\alpha[X_\alpha]/(X_\alpha ^r).
$$
\end{lemma}

\begin{proof}
First of all, we have
$$
E_{\Delta \backslash \{ \alpha \} }^{\rm sep}\otimes_{E_{\Delta\setminus\{\alpha\}}^+}E_{L,\Delta}^+/(X_\alpha ^r)=\varinjlim E_{{\beta _1}}^{\prime}{ \bigotimes \limits_{{E_{{\beta _1}}}}}\left(E_{{\beta _2}}^{\prime}{ \bigotimes \limits_{{E_{{\beta _2}}}}}\left( \cdots \left(E_{{\beta _{n-1}}}^{\prime}{ \bigotimes \limits_{{E_{{\beta _{n-1}}}}}}{E_{L,\Delta} }/(X_\alpha^r)\right)\right)\right).
$$
where $\Delta\setminus\{\alpha\}=\{\beta_1,\dots,\beta_{n-1}\}$.
As in Lemma \ref{Lem5.5} if $\varphi_\beta(u)=u$ for some element $u\in E_{\Delta \backslash \{ \alpha \} }^{\prime}\otimes_{E_{\Delta\setminus\{\alpha\}}^+}E_{L,\Delta}^+/(X_\alpha ^r)$ then by Lemma \ref{lem2.3} we must have $p{\rm val}_{X_\beta^\prime}(u)={\rm val}_{X_\beta^\prime}(u)$ showing ${\rm val}_{X_\beta\prime}(u)=0$. The constant term (with respect to the variables $X_\beta$, $\beta\in\Delta\setminus\{\alpha\}$) is an element in $k_{\Delta \backslash \{\alpha\}}\otimes_{\Bbb{F}_p} k_\alpha[X_\alpha]/(X_\alpha^r)$. So the statement follows the same way as in the proof of Lemma \ref{Lem5.5} noting 
$$
\bigcap\limits_{\beta\in \Delta \backslash \{\alpha\}} \left(k_{\Delta \backslash \{\alpha\}} \right)^{\varphi_\beta={\rm id}}=\Bbb{F}_p
$$
and $\varphi_\beta$ acts trivially on $k_\alpha[X_\alpha]/(X_\alpha^r)$. 
\end{proof}

\begin{lemma}\label{Lem5.11}
For any integer $r > 0$ and finitely generated $E_{\overline \alpha  }^ + /(X_\alpha ^r)$-module $M$ we have an
identification 
$$
E_{\Delta \backslash \{ \alpha \}}^{\rm sep}[X_\alpha ]/(X_\alpha ^r)\bigotimes \limits_{E_{\overline \alpha  }^ + /(X_\alpha ^r)} M \cong E_{\Delta \backslash \{ \alpha \} }^{\rm sep}\bigotimes \limits_{E_{\Delta \backslash \{ \alpha \} }}M.
$$
\end{lemma}

\begin{proof}
This follows from the isomorphism $E_{\overline \alpha  }^ + /(X_\alpha ^r) \cong {E_{\Delta \backslash \{ \alpha \} }}\otimes_{\Bbb{F}_p}k_\alpha[{X_\alpha }]/(X_\alpha ^r)$.
\end{proof}

For a subset $\Delta^\prime \subseteq \Delta$, we put $E_{\Delta^\prime}^{{\rm sep} +}: =\underrightarrow {\lim }E_{\Delta^\prime}^{\prime + }$
so we have ${E_{\Delta^\prime}^{\rm sep} = E_{\Delta^\prime}^{{\rm sep} + }[X_{\Delta^\prime}^{ - 1}]}$, where
$X_{\Delta^\prime}:=\prod \limits_{\alpha\in \Delta^\prime} X_\alpha$. 

The proofs of the following two Lemmas follow exactly as in \cite[Lemma 3.13]{Zabradi2} and \cite[Lemma 3.14]{Zabradi2} without 
any change.

\begin{lemma}\label{Lem5.12}
$E_{\Delta^\prime}^{\rm sep}$ (resp. ${E_{\Delta^\prime}^{{\rm sep} + }}$) is flat as a module over $E_{\Delta^\prime}$ (resp. over $E_{\Delta^\prime}^ + $
) for all $\Delta^\prime \subseteq \Delta$.
\end{lemma}

\begin{lemma}\label{Lem5.13}
We have $\left(E_{\Delta \backslash \{ \alpha \}}^{{\rm sep} + }\llbracket X_\alpha \rrbracket [X_\Delta^{-1}]\right)^{\mathcal{H}_{\Delta \backslash \{ \alpha \}}}  = E_{L,\Delta}$.
\end{lemma}

Our main result in this section is the following

\begin{theorem}\label{Them5.14}
The functors $\Bbb{D}$ and $\Bbb{V}$ are quasi-inverse equivalences of categories between
 ${\rm Rep}_{{\Bbb{F}_p}}(\mathcal{G}_{L,\Delta})$ and $\mathcal{D}^{\rm et}(\varphi _\Delta, G_{L,\Delta}, E_{L,\Delta})$.
\end{theorem}

\begin{proof}
The proof is essentially the same as that of \cite[Theorem 3.15]{Zabradi2}. Instead of repeating the proof, 
we recall the strategy and just point out, what the changes should be adapted to 
making the argument work in our imperfect residue field case.

{\bf Step1}. \emph{Reducing the statement to the essential surjectivity of $\Bbb{D}$}. 
By Theorem \ref{Them5.9} $\Bbb{D}$ is fully faithful and by Proposition \ref{Prop5.7} we have $\Bbb{V}\circ\Bbb{D}(V)\cong V$ for 
any object $V$ in ${\rm Rep}_{{\Bbb{F}_p}}(\mathcal{G}_{L,\Delta})$. So we are reduced to showing that 
$\mathbb{D}$ is essentially surjective. This is done by induction on $|\Delta|$. 
The case of $|\Delta|=1$ is due to Scholl \cite{Scholl} and Andreatta \cite[Theorem 7.11]{Andreatta}. 
Suppose that $|\Delta | > 1$, fix
${\alpha  \in \Delta }$, and pick an object $D$ in $\mathcal{D}^{et}(\varphi _\Delta, G_{L, \Delta}, E_{L, \Delta})$.

\vspace{2mm}

{\bf Step 2}. 
\emph{The goal here is to trivialize the $({\varphi_{\overline \alpha}, \varphi _\beta })$-action {\rm (}$\beta  \in \Delta \backslash \{ \alpha \} ${\rm )} on $D_{\overline \alpha  }^{ + *}/X_\alpha ^rD_{\overline \alpha  }^{ + *}$ uniformly
in $r$ by tensoring up with ${E_{\Delta \backslash \{ \alpha \} }^{sep}}$.} 
The place which we need to modify is the Equation (4) before \cite[Lemma 3.17]{Zabradi2}.
In our imperfect case,  it should be reformulated as
$$
\begin{aligned}
E_{\Delta \backslash \{ \alpha \} }^{\rm sep} [X_\alpha]/(X_\alpha ^r) & \bigotimes \limits_{\Bbb{F}_p[X_\alpha]/(X_\alpha ^r)}\, \,  \bigcap\limits_{\beta  \in \Delta \backslash \{ \alpha \} }\left (E_{\Delta \backslash \{ \alpha \} }^{\rm sep}[X_\alpha]/(X_\alpha ^r) \bigotimes \limits_{E_{\overline \alpha }^ + /(X_\alpha ^r)} D_{\overline \alpha  ,r}^{ + \ast} \right)^{\varphi_{\overline{\alpha}}={\rm id}, \, \, \varphi _\beta ={\rm id}} \\
&\, \, \,   \, \, \, \, \, \xrightarrow[]{\sim}   E_{\Delta \backslash \{ \alpha \} }^{\rm sep}[X_\alpha]/(X_\alpha ^r)\bigotimes \limits_{E_{\overline \alpha}^+/(X_\alpha ^r)} D_{\overline \alpha, r}^{+ \ast} \\
&\, \, \,  \, \, \, \, \,   \cong E_{\Delta \backslash \{ \alpha \} }^{\rm sep}[X_\alpha]/(X_\alpha ^r) \bigotimes _{E_{\overline \alpha }^ +} D_{\overline \alpha  }^{ + *}.
\end{aligned}
$$

Lemma 3.17 in \cite{Zabradi2} should be changed into: there exists a 
finitely generated $E_{L,\Delta} ^ + $-submodule $M \leqslant D_{\overline \alpha  }^{+ \ast}$
such that
$$
\bigcap \limits_{\beta  \in \Delta \backslash \{ \alpha \} } \left( E_{\Delta \backslash \{ \alpha \} }^{\rm sep}[X_\alpha]/(X_\alpha ^r)\bigotimes \limits_{E_{\overline \alpha  }^ + } 
D_{\overline \alpha  }^{+ \ast}\right)^{\varphi_{\overline{\alpha}}={\rm id}, \varphi _\beta = {\rm id}}
$$
is contained in the image of the map
$$
\begin{aligned}
E_{\Delta \backslash \{ \alpha \} }^{\rm sep}[X_\alpha]/(X_\alpha ^r) \bigotimes \limits_{E_{L,\Delta} ^ + } M & \longrightarrow E_{\Delta \backslash \{ \alpha \} }^{\rm sep}[X_\alpha]/(X_\alpha ^r) \bigotimes \limits_{E_{L,\Delta} ^ + }D_{\overline \alpha }^{ + \ast} \\
& \cong E_{\Delta \backslash \{ \alpha \} }^{\rm sep}[X_\alpha]/(X_\alpha ^r) \bigotimes \limits_{E_{\overline \alpha  }^ + }D_{\overline \alpha  }^{ + \ast}
\end{aligned}
$$
induced by the inclusion $M \leqslant D_{\overline \alpha  }^{ + *}$
for all $r > 0$.  

Note that the intersection in the above-mentioned is not only over $\varphi_\beta={\rm id}\, \, (\beta\in \Delta\backslash \{\alpha\})$ but 
also over $\varphi_{\overline{\alpha}}={\rm id}$, noting that $\varphi_{\overline{\alpha}}$ is the absolute Frobenius of 
$E_{\Delta \backslash \{\alpha\}}^{\rm sep}$. The reason why we need to make this change is because it 
is coherent with our Lemma \ref{Lem5.5} and also in the proof of \cite[Lemma 3.17]{Zabradi2} (where the third author 
use the fact that $\varphi^{l_r}_{\overline{\alpha}}(x)=x$).
\vspace{2mm}

{\bf Step 3}. \emph{The goal of this step is to show the following compatibility of our construction with projective limits with respect to $r$}. 
In our imperfect residue field case, Lemma 3.18 in \cite{Zabradi2} should be changed into: we have
$$
\varprojlim_r \left(E_{\Delta \backslash \{ \alpha \} }^{{\rm sep} + }[X_\alpha]/(X_\alpha ^r)\bigotimes \limits_{E_{L,\Delta}^+} M \right) \cong 
E_{\Delta \backslash \{ \alpha \} }^{{\rm sep} + }\llbracket X_\alpha \rrbracket \bigotimes \limits_{E_{L,\Delta} ^+} M, 
 $$
 $$
\varprojlim_r \left(E_{\Delta \backslash \{ \alpha \} }^{\rm sep}[X_\alpha]/(X_\alpha ^r) \bigotimes \limits_{E_{\overline \alpha}^ + } D_{\overline \alpha  }^{ + \ast} \right) \cong E_{\Delta \backslash \{ \alpha \} }^{\rm sep}\llbracket X_\alpha \rrbracket \bigotimes \limits_{E_{\overline \alpha  }^ + } D_{\overline \alpha  }^{ + *}, 
 $$
 and
 $$
 \begin{aligned}
&  \varprojlim_r \left(E_{\Delta \backslash \{ \alpha \} }^{\rm sep} [X_\alpha]/(X_\alpha ^r)\bigotimes \limits_{\Bbb{F}_p[X_\alpha]/(X_\alpha ^r)} \bigcap \limits_{\beta  \in \Delta \backslash \{ \alpha \} } \left( E_{\Delta \backslash \{ \alpha \} }^{\rm sep}[X_\alpha]/(X_\alpha ^r) \bigotimes \limits_{E_{\overline \alpha  }^ + /(X_\alpha ^r)} D_{\overline \alpha  ,r}^{ + *} \right)^{\varphi_{\overline{\alpha}}={\rm id}, \varphi _\beta= {\rm id}} \right)  \\
&  \cong  E_{\Delta \backslash \{ \alpha \} }^{\rm sep} \llbracket X_\alpha \rrbracket \bigotimes \limits_{\Bbb{F}_p \llbracket X_\alpha \rrbracket} \bigcap \limits_{\beta  \in \Delta \backslash \{ \alpha \} } \left(E_{\Delta \backslash \{ \alpha \} }^{\rm sep}\llbracket X_\alpha \rrbracket 
 \bigotimes \limits_{E_{L,\Delta}} D \right)^{\varphi_{\overline{\alpha}}={\rm id}, \varphi _\beta= {\rm  id}}.
 \end{aligned}
$$
 Everything else including the proof of \cite[Lemma 3.18]{Zabradi2} remains the same. 
\vspace{2mm}

{\bf Step 4}. 
\emph{Step 4. The goal here is to obtain a $(\varphi _\alpha, G_\alpha)$-module ${D_\alpha }$ over ${E_\alpha }$ {\rm (}by trivializing the action
of $(\varphi_{\overline{\alpha}}, \varphi _\beta),\beta  \in \Delta \backslash \{ \alpha \} ${\rm )} which is at the same time a linear representation of the group
$\mathcal{G}_{L, \Delta \backslash \{ \alpha \} }$.} 
In the Step 4 of the third author's proof (before \cite[Lemma 3.19]{Zabradi2}), the 
corresponding $D_\alpha$ in our imperfect residue field case should be defined as 
$$
D_\alpha: = \bigcap \limits_{\beta  \in \Delta \backslash \{ \alpha \} } \left(E_{\Delta \backslash \{ \alpha \} }^{\rm sep}((X_\alpha))\bigotimes \limits_{E_{L,\Delta} } D \right)^{\varphi_{\overline{\alpha}}={\rm id}, \varphi _\beta= {\rm id}}, 
$$
which is contained in the image of the map
$$
E_{\Delta \backslash \{ \alpha \} }^{{\rm sep} + } \llbracket X_\alpha \rrbracket [X_\alpha ^{ - 1}] \bigotimes \limits_{E_{L,\Delta}} D \hookrightarrow 
E_{\Delta \backslash \{ \alpha \} }^{\rm sep}((X_\alpha)) \bigotimes \limits_{E_{L,\Delta}} D.
 $$
Then $D_\alpha$ is an $\Bbb{F}_p((X_\alpha))$ vector space. 
\vspace{2mm}

{\bf Step 5}. \emph{The last step is to show the essential surjectivity of $\Bbb{D}$}. 
Lemma 3.19 of \cite{Zabradi2} exactly remains the same, including its proof, and so does 
the third author's Step 5 (cf. see after the proof of \cite[Lemma 3.19]{Zabradi2}). We do  not need to 
change anything here. 

\end{proof}

\begin{corollary}\label{Coro5.15}
Any object $D$ in $\mathcal{D}^{\rm et}({\varphi _\Delta, \varphi_s, G_{L,\Delta}, E_{L,\Delta} })$ is a free module over $E_{L,\Delta}$.
\end{corollary}
\begin{proof}
By Theorem \ref{Them5.14} we know that $\Bbb{D}$ is essentially surjective. The Corollary follows by noting that 
any \' etale module in the image of the functor $\Bbb{D}$ is free as a module over $E_{L,\Delta}$ by construction. 
\end{proof}

\noindent{\bf Acknowledgements}
We would like to thank the anonymous referees for his/her tremendous job, who helped to correct 
a number of errors, lacunae and other inaccuracies 
(both mathematical and pedagogical), also taught us some theory of $(\varphi, \Gamma)$-modules 
and representations of Galois groups. 

The first author and the third author are thankful to Beijing Institute of Technology (BIT) for its gracious 
hospitality during a visit on August 2019 when this collaboration took place. The main idea of this work 
grew out of their visit to BIT. The first author also acknowledges the grant from PIMS-CNRS postdoctoral fellowship from the University of British Columbia. 
The first author is supported by postdoc research fellowship from the Tata Institute of Fundamental Research 
during the final revision of this paper. The third author was supported by the R\'enyi Institute Lend\"ulet Automorphic Research Group, by the 
NKFIH Research Grant FK-127906, and by Project ED 18-1-2019-0030 (Application-specific highly reliable 
IT solutions) under the Thematic Excellence Programme funding scheme.

\addtocontents{toc}{
    \protect\settowidth{\protect\@tocsectionnumwidth}{}%
    \protect\addtolength{\protect\@tocsectionnumwidth}{0em}}

\end{document}